\documentclass{article}
\usepackage{algorithm}
\usepackage[noend]{algorithmic}
\newcommand{\address}[1]{\thanks{#1}}
\newcommand{\orgdiv}[1]{#1}
\newcommand{\orgname}[1]{#1}
\newcommand{\orgaddress}[1]{#1}
\newcommand{\street}[1]{#1}
\newcommand{\postcode}[1]{#1}
\newcommand{\country}[1]{#1}
\newcommand{\State}{\STATE}
\newcommand{\For}{\FOR}
\newcommand{\EndFor}{\ENDFOR}
\newcommand{\If}{\IF}
\newcommand{\EndIf}{\ENDIF}
\newcommand{\Else}{\ELSE}
\newcommand{\ElsIf}{\ELSIF}
\newcommand{\Return}{\RETURN}
\newcommand{\While}{\WHILE}
\newcommand{\EndWhile}{\ENDWHILE}
\usepackage{fullpage}
\usepackage{amsmath,amssymb,amsthm}
\usepackage{graphicx}
\usepackage{hyperref}

\theoremstyle{thmstyletwo}
\newtheorem{proposition}{Proposition}
\newtheorem{lemma}{Lemma}
\newtheorem{theorem}{Theorem}

\theoremstyle{definition}
\newtheorem{remark}{Remark}

\DeclareMathOperator{\prox}{prox}
\DeclareMathOperator{\dist}{dist}
\DeclareMathOperator{\dom}{dom}
\DeclareMathOperator{\proj}{proj}

\begin{document}

% IMA JNA

%\DOI{DOI HERE}
%\copyrightyear{20XX}
%\vol{00}
%\pubyear{20XX}
%\access{Advance Access Publication Date: Day Month Year}
%\appnotes{Paper}
%\copyrightstatement{Under review Oxford University Press}
%\firstpage{1}

%\corresp[*]{Corresponding author: \href{email:olivier.fercoq@telecom-paris.fr}{olivier.fercoq@telecom-paris.fr}}

%\received{Date}{0}{Year}
%\revised{Date}{0}{Year}
%\accepted{Date}{0}{Year}
%\keywords{Primal-dual algorithm; Adaptive step sizes; Coordinate descent.}

%\authormark{Author Name et al.}

%[Monitoring convergence speed of PDHG]
\title{Monitoring the Convergence Speed of PDHG \\ to Find Better Primal and Dual Step Sizes}

\author{Olivier Fercoq%*
		\address{\orgdiv{LTCI}, \orgname{Télécom Paris, Institut Polytechnique de Paris}, \orgaddress{\street{Palaiseau}, \postcode{91120}, \country{France}},
		\texttt{olivier.fercoq@telecom-paris.fr}}
}

%\editor{Associate Editor: Name}

\maketitle
\abstract{Primal-dual algorithms for the resolution of convex-concave saddle point problems usually come with one or several step size parameters. Within the range where convergence is guaranteed, choosing well the step size can make the difference between a slow or a fast algorithm. A usual way to adaptively set step sizes is to ensure that there is a fair balance between primal and dual variable's amount of change.
	In this work, we show how to find even better step sizes for the primal-dual hybrid gradient. Getting inspiration from quadratic problems, we base our method on a spectral radius estimation procedure and try to minimize this spectral radius, which is directly related to the rate of convergence. Building on power iterations, we could produce spectral radius estimates that are always smaller than 1 and work also in the case of conjugate principal eigenvalues. 
	For strongly convex quadratics, we show that our step size rule yields an algorithm as fast as inertial gradient descent. Moreover, since our spectral radius estimates only rely on residual norms, our method can be readily adapted to more general convex-concave saddle point problems. 
	In a second part, we extend these results to a randomized version of PDHG called PURE-CD. We design a statistical test to compare observed convergence rates and decide whether a step size is better than another. 
	Numerical experiments on least squares, sparse SVM, TV-L1 denoising and TV-L2 denoising problems support our findings. }

\section{Introduction}

Primal-dual algorithms for the resolution of convex-concave saddle point problems, like for instance Primal-Dual Hybrid Gradient~\cite{chambolle2011first}, usually come with one or several step size parameters. Within the range where convergence is guaranteed, choosing well the step size can make the difference between a slow or a fast algorithm. Common knowledge states that good step sizes should lead to a fair balance between primal and dual variable's amount of change.

However, if this kind of expert tuning works well when routinely solving similar problems, it has obvious limits. First, the meaning of a fair balance is not quantitatively defined. 
Second, this does not help much when designing a general purpose solver like \cite{applegate2021practical}. A natural solution is to set the step sizes adaptively by monitoring key quantities about the iterates. A notable example is given by \cite{goldstein2015adaptive} where the step sizes are modified in order to balance the size of primal and dual residuals. However, despite indisputable numerical experiments, they could only prove convergence of the adaptive algorithm using an artificial slow down of the step size updates. In particular they do not theoretically show that the adaption is useful, even for a given restricted class of problems. This work has been extended recently to the case of stochastic PHDG in \cite{chambolle2023stochastic}.
When the Lagrangian function is strongly convex-concave, \cite{islamov2021pdhg} developed a method that provably adapts the step size. However, it requires the knowledge of the strong convexity constant, which is quite restrictive in the context of primal-dual methods.

We propose here new insights for the adaptive setting of step sizes. 
We show that on quadratic problems, it is possible to find better step sizes than Goldstein et al.'s \cite{goldstein2015adaptive}. We base our method on a spectral radius estimation procedure and try to minimize this spectral radius, which is directly related to the rate of convergence. Building on power iterations~\cite{hager2021applied}, we could produce spectral radius estimates that are always smaller than 1, work in the case of conjugate pairs of principal eigenvalues and induce a minor computational overhead. 

For strongly convex quadratics, we show that our step size rules yields an algorithm as fast as inertial gradient descent~\cite{nesterov2003introductory}. Moreover, since our spectral radius estimates only rely on residual norms, our method can be readily adapted to more general convex-concave saddle point problems. 

Other contributions include an adaptation of Goldstein et al.'s adaptive step sizes to Tri-PD~\cite{latafat2018primal} and the combination of Goldstein et al.'s step sizes with ours. Indeed, the former is effective from the first iterations while ours takes some time before being able to decide a change in step sizes.

In order to solve large scale problems, it is much more efficient to rely on a randomized algorithm, like S-PDHG~\cite{chambolle2018stochastic} or PURE-CD~\cite{alacaoglu2020random}.
These algorithms only update one primal and some of the dual variables yielding much cheaper iterations. Moreover, the conditions on the step sizes are quite favorable and after seeing all the data, much more progress has been done towards the solution set. 
However, like their deterministic counterpart, current theory is not giving much insight on how to precisely set their step sizes. 

We first revisit the residual balance strategy proposed in \cite{chambolle2023stochastic} for S-PDHG and adapt it to PURE-CD. Then, we propose a way to monitor the convergence of the algorithm despite the stochastic nature of the algorithm. We propose two models for the observed instantaneous rate: either independent and identically distributed or autoregressive of order 1. In both cases, we design a statistical test and accept to change the step size only if the new one gives a statistically significant improvement in the rate.

Like previous works~\cite{goldstein2013adaptive,chambolle2023stochastic}, our convergence results only provide conditions under which convergence can be retained with adaptive steps.
However, our results allow an asymptotically unbounded change in the amplitude of the step sizes, which was not possible with previous approaches. 
Our proof technique bases on the smoothed duality gap~\cite{tran2018smooth,walwil:hal-04501394}, which allowed us to show convergence of the feasibility gap and optimality gap without basing on a given constant norm to measure distances.

We conclude the paper with numerical experiments on least squares, sparse SVM, TV-L1 denoising and TV-L2 denoising. In the four cases, we make similar conclusions. Residual balance techniques for deterministic or randomized versions of PDHG very quickly identifies good step sizes. Indeed, they usually outperform a general purpose initialization that we set before observing the data. However, although convergence monitoring takes more time to identify good step sizes, the final value is much better than with residual balance.
The challenge is to find accurate estimates of the convergence rate by observing only a few iterations in order to be able to discriminate between two set of step sizes.

\section{Problem, notation and primal-dual algorithm}

We consider the composite optimization problem $\min_{x \in \mathbb R^n} f(x) + f_2(x) + g(Ax)$ where $A$ is a $m \times n$ matrix, $f$ and $g$ are convex lower semi-continuous functions whose proximal operator is easily computable and $f_2$ is a convex differentiable function whose gradient is $L_f$-Lipschitz continuous.
To solve it, we shall resort to the equivalent saddle point problem
\begin{equation}
\min_{x\in \mathbb R^n} \max_{y \in \mathbb R^m} f(x) + f_2(x) + \langle Ax, y\rangle - g^*(y) 
\label{pb:saddle_point}
\end{equation}
where $g^*$ is the Fenchel conjugate of $g$. The qualification condition $0 \in \mathrm{ri}(\dom g - A(\dom f))$, that we shall assume throughout, is enough to guarantee that the values of both problems are equal~\cite{bauschke2011convex}.

Primal-Dual Hybrid Gradient is an algorithm that is able to solve this problem by using only matrix-vector products, proximal operators and gradients as elementary operations.
We shall consider two versions of PDHG, that are equivalent when $f_2 = 0$~\cite{condat2023proximal}.
We will use the letter $x$ for primal variables, $y$ for dual variables and $z = (x,y)$ when working in primal and dual space together.

\subsection{V\~{u}-Condat algorithm}

For step size parameters $\sigma_k >0$ and $\tau_k > 0$ that satisfy $\sigma_k \tau_k \|A\|^2 + \tau_k L_f/2 < 1$ for all $k$, the  algorithm is given by

\begin{algorithm}
 $x_0 \in \mathbb R^n$, $y_0 \in \mathbb R^m$ and for all $k \geq 0$:
\begin{align*}
&y_{k+1} = \prox_{\sigma_k g^*}(y_k + \sigma_k A x_k) \\
&x_{k+1} = \prox_{\tau_k f} (x_k - \tau_k \nabla f_2(x_k) - \tau_k A^\top (2 y_{k+1} -y_k)) 
\end{align*}
\caption{V\~{u}-Condat algorithm}
\label{alg:vu-condat}
\end{algorithm}

Denote $T_{VC}$ the operator such that $z_{k+1} = T_{VC}(z_k)$. It has been shown in~\cite{vu2013splitting} and \cite{condat2013primal} that $T_{VC}$ is an averaged operator when considering the norm $\|\cdot\|_{V_k}$ defined by $\|z\|^2_{V_k} = \frac 1 {\tau_k} \|x\|^2 + 2 \langle Ax, y\rangle + \frac 1 {\sigma_k} \|y\|^2$. This implies the convergence of the algorithm in the constant step size case.

\subsection{Latafat et al.'s Tri-PD algorithm}

For step size parameters $\sigma_k >0$ and $\tau_k > 0$ that satisfy $\sigma_k \tau_k \|A\|^2 < 1$  and $\tau_k L_f < 2$, the  algorithm is given by 
\begin{algorithm}
$x_0 \in \mathbb R^n$, $y_0 \in \mathbb R^m$ and for all $k \geq 0$:
\begin{align*}
&\bar x_{k+1} = \prox_{\tau_k f} (x_k - \tau_k \nabla f_2(x_k) - \tau_k A^\top y_{k}) \\ 
&y_{k+1} = \bar y_{k+1} = \prox_{\sigma_k g^*}(y_k + \sigma_k A \bar x_{k+1}) \\
& x_{k+1} = \bar x_{k+1} - \tau_k A^\top (y_{k+1} - y_k)
\end{align*}
\caption{Tri-PD algorithm}
\label{alg:tri-pd}
\end{algorithm}

Denote $T_L$ the operator such that $z_{k+1} = T_L(z_k)$. It has been shown in~\cite{latafat2018primal} that $T_L$ is an averaged operator when considering the norm $\|\cdot\|_{V_k}$ defined by $\|z\|^2_{V_k} = \frac 1 {\tau_k} \|x\|^2 + \frac 1 {\sigma_k} \|y\|^2$, which implies the convergence of the algorithm in the constant step size case.

In the case $f_2=0$ and $\tau_k = \tau$, both algorithms are equivalent (up to a change of indices) because $x_k - \tau A^{\top} y_k = \bar x_k - \tau A^\top (2 \bar y_k - \bar y_{k-1})$.

\subsection{Linear convergence rate}

\label{sec:linconv}

Suppose that the Lagrangian function $L$ defined by $L(x,y) = f(x) + f_2(x) + \langle Ax, y\rangle - g^*(y)$ satisfies that $(x \mapsto L(x,y))$ is strongly convex for all $y$ with constant $\mu_f$ and $(y \mapsto L(x,y))$ is strongly concave for all $x$ with constant $\mu_{g^*}$. In that case, both versions of PDHG converge linearly (see \cite{fercoq2019coordinate} for Vu-Condat and \cite{fercoq2022quadratic} for Tri-PD). Yet, the rate depends on the strong convexity constants: there exists $c > 0$ that depends only on the algorithm such that for all $k$,
\begin{equation}
\|z_{k+1} - z_*\|^2_V \leq \big(1 - c\min(\tau \mu_f, \sigma \mu_{g^*}) \big)\|z_k - z_*\|^2 \;. 
\label{eq:rate_strong_convex}
\end{equation}
Hence, if we know the values of $\mu_f$ and $\mu_{g^*}$, we can choose the step sizes $\sigma$ and $\tau$ that would maximize $\min(\tau \mu_f, \sigma \mu_{g^*})$ while satisfying the constraints guaranteeing convergence.

Note that the resulting algorithm is indeed quite good. In the case where $f_2=0$, the only constraint is $\sigma\tau \|A\|^2 < 1$, so that we may choose $\sigma = \sqrt{\frac{ \gamma\mu_f}{\mu_{g^*}\|A\|^2}}$ and $\tau = \sqrt{\frac{ \gamma\mu_{g^*}}{\mu_{f}\|A\|^2}}$ for $\gamma < 1$ but close to 1. We obtain
$\min(\tau \mu_f, \sigma \mu_{g^*}) = \sqrt{\frac{ \gamma\mu_f\mu_{g^*}}{\|A\|^2}} \approx \sqrt{\frac{\mu_f\mu_{g^*}}{\|A\|^2}}$, which is optimal for this class of problems \cite{nesterov2003introductory}.

Yet, PDHG does converge linearly for much more general classes of problems. In particular, for all piecewise linear-quadratic problems, which include linear programs and quadratic programs, PDHG enjoys linear convergence. However, the constants are usually unknown and the influence of the step size on the rate is not well understood.

\section{Adaptive step sizes based on residual balance}

\subsection{Goldstein et al.'s adaptive step sizes}

In~\cite{goldstein2013adaptive} and \cite{goldstein2015adaptive}, Goldstein, Li, Yuan, Esser and Baraniuk proposed an adaptive way of setting the step sizes for Chambolle and Pock's algorithm, i.e. Algorithm~\ref{alg:vu-condat} in the case $f_2=0$. 
Their intuition is that good primal and dual step sizes should balance the progress between the primal and dual space. 
Our goal is to find a saddle point of $L$, that is a point $z_* = (x_*, y_*)$ such that
\begin{align*}
&0 \in \partial f(x_*) + A^\top y_* \\
&0 \in \partial g^*(y_*) - Ax_*
\end{align*}

Said otherwise, if for $z = (x,y)$, we denote $F(z) = (\partial f(x) + A^\top y) \times (\partial g^*(y) - Ax)$, our goal is to find $z$ such that $\|F(z)\|_0 = \min \{\|q\|, q \in F(z)\}$ is as small as possible~\cite{lu2022infimal}. 
Goldstein et al.\ first remarked that one can easily find a point in $F(z_{k+1})$. From the definition of Algorithm~\ref{alg:vu-condat}, we have
\begin{align*}
& 0 \in \sigma_k \partial g^*(y_{k+1}) + y_{k+1} - y_k - \sigma_k A x_k\\
& 0 \in \tau_k \partial f(x_{k+1}) + x_{k+1} - x_k + \tau_k A^\top (2 y_{k+1} - y_k)
\end{align*}
so that
\begin{align*}
d_{k+1} = & \frac{1}{\sigma_k} (y_k - y_{k+1}) + A(x_k - x_{k+1}) \in \partial g^*(y_{k+1}) - A x_{k+1}\\
p_{k+1} = & \frac{1}{\tau_k} (x_k - x_{k+1}) + A^\top (y_k - y_{k+1})  \in \partial f(x_{k+1}) + A^\top y_{k+1} 
\end{align*}

Then, we can compare $\|d_{k+1}\|_1$ and $\|p_{k+1}\|_1$. If $\|p_{k+1}\|_1 > \Delta \|d_{k+1}\|_1$ for some $\Delta > 1$, this means that the primal space is given too much importance and we should decrease $\tau_k$ (and increase $\sigma_k$ accordingly). Similar considerations hold if the unbalance is opposite. The choice of the 1-norm is motivated experimentally.
In the end, we obtain Algorithm~\ref{alg:goldstein}.

\begin{algorithm}
	\begin{algorithmic}
		\State Parameters: $\underline \alpha = 10^{-4}, \eta = 0.95, \Delta=1.5$
		\If{$\alpha_k \leq \underline \alpha$}
		\State Skip step size adaptation: $\tau_{k+1} = \tau_k$, $\sigma_{k+1} = \sigma_k$, $\alpha_{k+1} = \alpha_{k}$
		\Else  
		\State $p_{k+1} = \frac{1}{\tau_k} (x_k - x_{k+1}) - A^\top (y_k - y_{k+1})$
		\State $d_{k+1} = \frac{1}{\sigma_k} (y_k - y_{k+1}) - A(x_k - x_{k+1})$ 
		\If{$\|p_{k+1}\|_1 \geq \Delta \|d_{k+1}\|_1$}
		\State $\tau_{k+1} = \tau_k / (1-\alpha_k)$, $\sigma_{k+1} = \sigma_k (1-\alpha_k)$, $\alpha_{k+1} = \alpha_k \eta$
		\ElsIf{$\|d_{k+1}\|_1 \geq \Delta \|p_{k+1}\|_1$}
		\State $\tau_{k+1} = \tau_k (1-\alpha_k)$, $\sigma_{k+1} = \sigma_k / (1-\alpha_k)$, $\alpha_{k+1} = \alpha_k \eta$
		\Else
		\State $\tau_{k+1} = \tau_k$, $\sigma_{k+1} = \sigma_k$, $\alpha_{k+1} = \alpha_k$
		\EndIf
		\EndIf
		\Return $\tau_{k+1}, \sigma_{k+1}, \alpha_{k+1}$
	\end{algorithmic}
	\caption{Goldstein et al.'s adaptive steps -- Goldstein$(z_k, z_{k+1}, \tau_k, \sigma_k, \alpha_k)$}
	\label{alg:goldstein}
\end{algorithm}

The algorithm includes a geometric slow-down in the updates. This ensures that the adaptive step size rule will not prevent convergence because the step sizes will eventually be nearly constant. Moreover, we can show that $\forall k,$
\begin{align*}
\tau_k& \leq \tau_0 \Big(\prod_{i=0}^{k-1} (1-\alpha_i)\Big)^{-1}
= \tau_0 \Big(\prod_{i=0}^{k-1} (1-\alpha_0 \eta^i)\Big)^{-1} = \tau_0 \exp\Big(-\sum_{i=0}^{k-1} \log(1-\alpha_0\eta^i)\Big) \\
&\leq \tau_0 \exp \Big(-\sum_{i=0}^{k-1}\log(1-\alpha_0) \eta^i \Big)
= \tau_0 \exp\Big(-\log(1-\alpha_0)\frac{1-\eta^k}{1-\eta} \Big) \leq \tau_0 \exp\Big(-\frac{\log(1-\alpha_0)}{1-\eta} \Big) \\
& \leq \tau_0 (1-\alpha_0)^{-\frac {1}{1-\eta}}
\end{align*}
where we used the inequality $\log(1-x) \geq \frac{x}{x_0} \log(1-x_0)$ valid for all $x \in [0, x_0]$ by concavity of the logarithm.
We can prove similarly that $\tau_k \geq \tau_0 (1-\alpha_0)^{\frac {1}{1-\eta}}$ and similar results for $\sigma_k$. With $\alpha_0 = 0.5$ and $\eta = 0.95$, we obtain $(1-\alpha_0)^{-\frac {1}{1-\eta}} \approx 10^6$. Hence, this gives a large updating power to the method but at the same time prevents a race without end if we ever encounter a pathological cases.

Let us now give an example where we can show optimality of Goldstein et al.'s adaptive step sizes.
\begin{proposition}
Consider the toy problem
\[
\min_{x \in \mathbb R^n, x^0 \in \mathbb R} \max_{y \in \mathbb R^m, y^0 \in \mathbb R} f(x) - g^*(y) + \iota_{\{0\}}(x^0) + x^0 y^0 - \iota_{\{0\}}(y^0)
\]
where $f$ in $\mu_f$-strongly convex and $g^*$ is $\mu_{g^*}$-strongly convex
and suppose we are solving it with Algorithm \ref{alg:vu-condat} with constant step sizes $\tau$ and $\sigma$ such that $\sigma \tau = \gamma < 1$.

Whatever the value of $\Delta$,
if $\tau >  \sqrt{\frac{ \gamma\mu_{g^*}}{\mu_{f}}}$, then eventually $\Delta \|p_k\|_1 < \|d_k\|_1$ and if  $\tau < \sqrt{\frac{ \gamma\mu_{g^*}}{\mu_{f}}}$, then eventually $\Delta \|d_k\|_1 < \|p_k\|_1$.
\end{proposition}
\begin{proof}
The convex-concave function has a first part where primal and dual contributions are completely decoupled and a second part which is trivially solved with $(x^0, y^0) = (0,0)$ but constrains the step sizes to satisfy $\sigma \tau = \gamma < \|A\|^2 = 1$.
Hence, on this problem, Algorithm \ref{alg:vu-condat} is equivalent to the proximal point method run in parallel on $f$ and $g^*$. This implies that $\|p_k\|_1$ will converge to 0 with a rate equal to $(1-\tau \mu_f)$ and  $\|d_k\|_1$ will converge to 0 with a rate equal to $(1-\sigma \mu_{g^*})$.

This means that, if $\tau > \sqrt{\frac{ \gamma\mu_{g^*}}{\mu_{f}}}$, then there will be some $k$ such that $\Delta \|p_k\|_1 < \|d_k\|_1$ and then $\tau$ should be decreased according to Algorithm~\ref{alg:goldstein}. 
Similarly, if $\tau < \sqrt{\frac{ \gamma\mu_{g^*}}{\mu_{f}}}$, then there will be some $k$ such that $\Delta \|d_k\|_1 < \|p_k\|_1$ and then $\tau$ should be increased according to Algorithm~\ref{alg:goldstein}. 
\end{proof}

As we will see in numerical experiments, Goldstein et al.'s adaptive step size is very efficient to detect big discrepancies between primal and dual residuals.
However, when the sequence of iterates oscillates between the primal and dual space, it cannot keep track of these oscillations and its damping factor automatically stops the adaptation of the step sizes. Moreover, the step sizes we obtain are usually rather good but far from being optimal in the long run.

\subsection{Generalization to the case $f_2 \neq 0$}

We now consider Algorithm \ref{alg:vu-condat} with $f_2 \neq 0$.
In this case, we want to find a zero of the opertor $F(z) = (\partial f(x) + \nabla f_2(x) + A^\top y) \times (\partial g^*(y) - Ax)$.
To estimate residual balance, we need a point in $F(z_{k+1})$. From the definition of Algorithm~\ref{alg:vu-condat}, we have
\begin{align*}
& 0 \in \sigma_k \partial g^*(y_{k+1}) + y_{k+1} - y_k - \sigma_k A x_k\\
& 0 \in \tau_k \partial f(x_{k+1}) + \tau_k \nabla f_2(x_k) + x_{k+1} - x_k + \tau_k A^\top (2 y_{k+1} - y_k)
\end{align*}
so that
\begin{align*}
d_{k+1} = & \frac{1}{\sigma_k} (y_k - y_{k+1}) + A(x_k - x_{k+1}) \in \partial g^*(y_{k+1}) - A x_{k+1}\\
p_{k+1} = & \frac{1}{\tau_k} (x_k - x_{k+1}) +\nabla f_2(x_{k+1}) - \nabla f_2(x_k) + A^\top (y_k - y_{k+1})  \in \partial f(x_{k+1})+ \nabla f_2(x_{k+1}) + A^\top y_{k+1} 
\end{align*}
With this new definition of primal and dual residuals, the rest of Algorithm~\ref{alg:goldstein} can be kept identical.

\subsection{Residual balance for Tri-PD}

Goldstein et al's primal and dual residual balance algorithm can also be defined for Algorithm~\ref{alg:tri-pd}. 
Since
\begin{align*}
& 0 \in \tau_k \partial f(\bar x_{k+1}) + \tau_k \nabla f_2(x_k) + \bar x_{k+1} - x_k + \tau_k A^\top y_k\\
& 0 \in \sigma_k \partial g^*(y_{k+1}) + y_{k+1} - y_k - \sigma_k A \bar x_{k+1}
\end{align*}
we can choose
\begin{align*}
p_{k+1} = & \frac{1}{\tau_k} (x_k - x_{k+1}) +\nabla f_2(\bar x_{k+1}) - \nabla f_2(x_k) \in \partial f(\bar x_{k+1})+ \nabla f_2(\bar x_{k+1}) + A^\top \bar y_{k+1}  \\
d_{k+1} = & \frac{1}{\sigma_k} (y_k - y_{k+1}) \in \partial g^*(\bar y_{k+1}) - A \bar x_{k+1}
\end{align*}
where we use the identities $x_{k+1} = \bar x_{k+1} - \tau_k A^\top (y_{k+1} - y_k)$
and $y_{k+1} = \bar y_{k+1}$. Hence $(p_{k+1}, d_{k+1}) \in F(\bar z_{k+1})$ and we can proceed with the primal and dual residual balance method. Note that in the case $f_2 \neq 0$, we need to compute $\nabla f_2(\bar x_{k+1})$, whereas when using Algorithm~\ref{alg:vu-condat}, all the quantities involved in the formula of the residuals are already required in order to run the algorithm.

\section{Quadratic case}
\label{sec:quadratic}

\subsection{Minimization of the spectral radius}

If the Lagrangian function is a quadratic function
\[
L(x,y) = \frac 12 x^\top Q x + c^\top x + y^\top A x - b^\top y - \frac 12 y^\top S y
\]
then the saddle point problem is equivalent to the resolution of a linear system of equations and PDHG can be written as
\begin{align*}
&y_{k+1} = (I + \sigma S)^{-1}(y_k + \sigma (A x_k - b)) \\
&x_{k+1} = (I + \tau Q)^{-1} (x_k - \tau (A^\top (2 y_{k+1} - y_k )+c)) \\
& \qquad \qquad = (I + \tau Q)^{-1} \Big(x_k - \tau \big(A^\top (2 (I + \sigma S)^{-1}[y_k + \sigma (A x_k - b)] - y_k )+c\big)\Big)
\end{align*}

Hence the algorithm is the fixed point algorithm
\[
z_{k+1} = R z_k + d
\]
where the matrix $R$ is given blockwise by
\begin{equation}
R = \begin{bmatrix}
(I+\tau Q)^{-1} (I - 2 \tau \sigma A^\top (I+\sigma S)^{-1} A) &
\tau (I+\tau Q)^{-1} A^\top (I - 2 (I+\sigma S)^{-1}) \\
\sigma (I+ \sigma S)^{-1} A & (I+ \sigma S)^{-1}
\end{bmatrix}
\label{eq:Rmatrix}
\end{equation}
and the vector $d$ is
\begin{equation}
d = \begin{bmatrix} -\tau (I+ \tau Q)^{-1} (-2 \sigma A^\top (I + \sigma S)^{-1}  b + c) \\
-\sigma (I + \sigma S)^{-1} b \end{bmatrix}
\end{equation}

The speed of convergence of the algorithm is governed by the spectral radius of the matrix $R$. In order to get a faster algorithm, we need to find step sizes $\tau$ and $\sigma$ that minimize this spectral radius. Even though this procedure is more costly than solving the original problem, it will consist in a goal for what can be expected from good step sizes.

\subsection{Estimation of the spectral radius using the power method}

\subsubsection{Unique principal eigenvalue}

If the spectral radius $|\lambda_1|$ of a matrix $B$ is supported by one single eigenvalue, then the power iteration method 
$x_{k+1} = B x_k$ satisfies $|\lambda_1| = \lim_{k \to +\infty} \frac{\|x_{k+1}\|}{\|x_k\|}$, where the result is true for any norm~\cite{hager2021applied}.
Note that when running PDHG on a quadratic problem, we have $z_{k+1} - z_k = R(z_k - z_{k-1})$. Hence, PDHG is producing a power sequence on $z_{k+1} - z_k$. Moreover, if 1 is an eigenvalue of $R$ (which may happen if $d=0$), this procedure automatically discards the associated eigenvectors.
We propose to use of the norm $\|\cdot\|_V$, for which nonexpansiveness is proved. This implies that the estimate of the spectral radius given by 
\[
\frac{\|z_{k+1} - z_k\|_V}{\|z_k - z_{k-1}\|_V} = |\lambda_1|\bigg(1+ O\Big(\Big(\frac{\lambda_2}{\lambda_1}\Big)^k\Big)\bigg)
\]
will always be smaller than 1. 
As we can see in Figure~\ref{fig:compareVand2norms}, this leads to much more stable estimates of the spectral radius, especially when the assumption of a single principal eigenvector is not true, a case that will be treated below.

\begin{figure}
	\centering 
	
\includegraphics[width=0.5\linewidth]{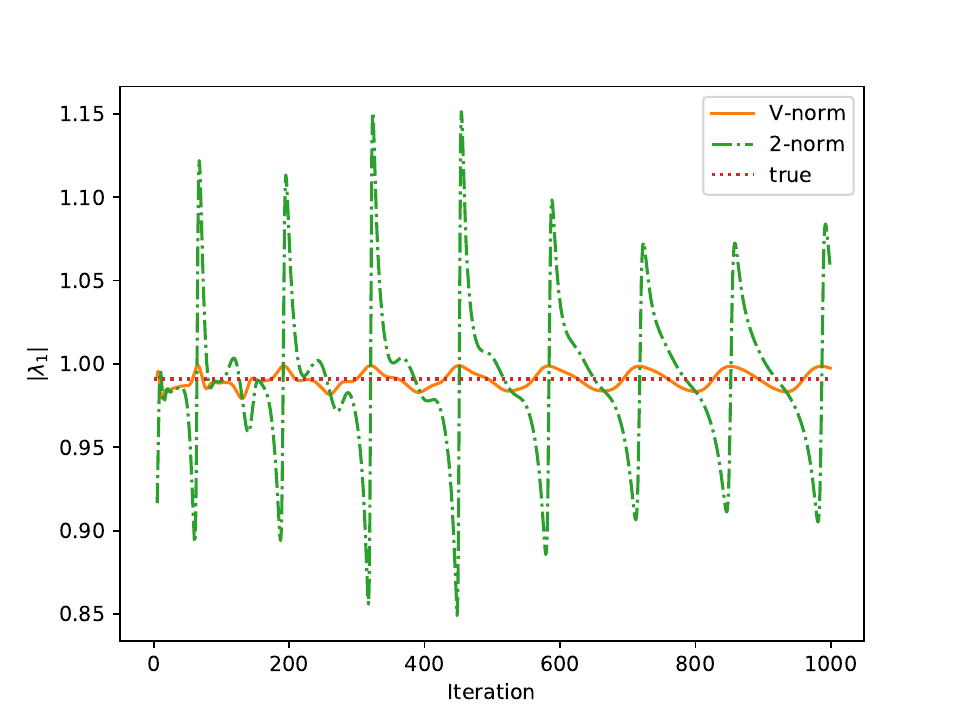}
\caption{
Comparison of norms for the computation of spectral radius estimates.
The problem under consideration is a least squares problem $\min_x \max_y \mu_x/2 ||x||^2 + \langle Ax, y \rangle - \mu_y/2 ||y||^2$ where each line of $A$ is such that $A_i x = (1+\eta) x_i - x_{i+1}$. We took $\mu_x = 0.01$, $\mu_y = 0.1$ and $\eta = 0.001$. The step sizes are constant with $\tau = 10 / \|A\|$. 
We can see that using the norm for which nonexpansiveness is guaranteed reduces a lot the amplitude of oscillations.
}
\label{fig:compareVand2norms}
\end{figure}

\subsubsection{Conjugate pair of principal eigenvalues}

When the spectral radius is supported by a conjugate pair of eigenvalues, the ratio $\frac{\|z_{k+1} - z_k\|_V}{\|z_k - z_{k-1}\|_V}$ may not converge.
As explained in~\cite{hager2021applied}, we can still use the power sequence to determine several eigenvalues, by a kind of Krylov method using $x_k$, $A x_k$ and $A^2 x_k$ instead of just $x_k$ and $A x_k$. However, this is quite sensitive to the actual number of principal eigenvalues: searching for a conjugate pair when $|\lambda_1| > |\lambda_2| = |\lambda_3|$ will lead to numerical issues. Moreover, there is no reason for the estimates to be always smaller than 1, even if they will be asymptotically.

%Asymptotically, $u_s = A^s u_0 = \alpha_1 x_1 + \alpha_2 x_2 + \alpha_3 x_3 + O((\lambda_4/\lambda_1)^s)$
%Let $p \in \mathbb R^3$ such that $(\lambda - \lambda_1)(\lambda - \lambda_2)(\lambda - \lambda_3) = \lambda^3 + p_2 \lambda^2 + p_1 \lambda + p_0$.
%We have $A^3 u_s + p_2 A^2 u_s + p_1 A u_s + p_0 u_s =  O((\lambda_4/\lambda_1)^s)$, hence, we can approximate $p$ by least squares.

We propose here a method to estimate the absolute value of a conjugate pair of principal eigenvalues using quantities computed using $\|\cdot\|_V$.

\begin{proposition}
Suppose that the spectral radius of a matrix $R$ is supported by a pair of conjugate eigenvalues $\lambda_1$ and $\lambda_2 = \bar \lambda_1$ associated with their (complex) eigenvectors $\zeta_1$ and $\zeta_2 = \bar \zeta_1$. We assume that $|\lambda_1| = |\bar \lambda_1| > |\lambda_j|$ for all $j \geq 3$.
Let $(u_k)$ be a sequence defined by $u_{k+1} = R u_k$. Then for any Hilbertian norm,
\[
\frac{\|u_{k+1}\|^2}{\|u_k\|^2} = \phi(k) + O\Big( (\lambda_3/\lambda_4)^k\Big)
\]
where the function $\phi$ is defined by $\phi(x) = |\lambda_1|^2 \Big(1 + b\frac{\cos(2 (x+1) \theta_1 + \varphi) - \cos(2 x \theta_1 + \varphi)}{a + b \cos(2 x \theta_1 + \varphi)} \Big)$, $\theta_1$ is the argument of $\lambda_1$ and $\varphi$, $a$ and $b$ are real numbers.

$\phi$ is periodic and, denoting $\underline x$ and $\bar x$ one of its minimum and maximum, we recover the spectral radius by $|\lambda_1|= \sqrt{\phi(\frac{\underline x + \bar x}{2})}$.
\end{proposition}

\begin{proof}
We write $u_0$ in the base of eigenvectors of $R$:
\[
u_0 = \sum_{j=1}^{m+n} \alpha_j \zeta_j
\]
Since $R \zeta_i = \lambda_i \zeta_i$, we have $u_k = \sum_{i=1}^{m+n} \alpha_i \lambda_i^k \zeta_i$. Denote $\lambda_1 = |\lambda_1| \exp(i \theta_1)$. Then
\begin{align*}
&u_k = |\lambda_1|^k \Big(2 \Re(\alpha_0 \exp(i k \theta_1) \zeta_1) + \sum_{j=3}^{m+n} O\Big(\Big(\frac{\lambda_j}{\lambda_1}\Big)^k\Big) \Big) \\
&\|u_k\| = |\lambda_1|^k \Big(2 \|\Re(\alpha_0 \exp(i k \theta_1) \zeta_1)\| + O\Big(\Big(\frac{\lambda_3}{\lambda_1}\Big)^k\Big)\Big)
\end{align*}

Note that 
\begin{align*}
4 \|\Re(\alpha_0 &\exp(i k \theta_1) \zeta_1)\|^2 = \langle \alpha_0 \exp(i k \theta_1) \zeta_1 + \bar\alpha_0 \exp(-i k \theta_1) \bar\zeta_1, \alpha_0 \exp(i k \theta_1) \zeta_1 + \bar\alpha_0 \exp(-i k \theta_1) \bar\zeta_1\rangle\\
& = \| \alpha_0 \exp(i k \theta_1) \zeta_1\|^2 + \|\bar\alpha_0 \exp(-i k \theta_1) \bar\zeta_1\|^2 + 2 \Re (\langle \alpha_0 \exp(i k \theta_1) \zeta_1, \bar\alpha_0 \exp(-i k \theta_1) \bar\zeta_1 \rangle) \\
& = 2 \| \alpha_0 \zeta_1\|^2 + 2 \Re(\exp(2 i k \theta_1) \langle \alpha_0 \zeta_1, \bar\alpha_0\bar\zeta_1 \rangle) 
\end{align*}
%Since $T$ is averaged, there exists $\alpha \in [0,1]$ such that $R = (1-\alpha) I + \alpha R_0$ and all the eigenvalues of $R_0$ are smaller than 1 in modulus.
%In particular, if $|\lambda_1| = 1 - \epsilon$, then $\cos(\theta_1) \geq 1 - \epsilon \frac{1-\epsilon/2}{1-\alpha}$. This implies that $\theta_1$ is likely to be rather small. 

We now consider the function 
\[
\phi(k) = |\lambda_1|^2 \frac{||\alpha_0 \zeta_1||^2 + |\langle \alpha_0 \zeta_1, \bar \alpha_0 \bar \zeta_1\rangle| \cos(2 (k+1) \theta_1 + \varphi)}{||\alpha_0 \zeta_1||^2 + |\langle \alpha_0 \zeta_1, \bar \alpha_0 \bar \zeta_1\rangle| \cos(2 k \theta_1 + \varphi)} = 
|\lambda_1|^2 \frac{a + b \cos(2(k+1) \theta_1 + \varphi)}{a + b \cos(2 k \theta_1 + \varphi)}\]
where $\varphi$ is the argument of $\langle \alpha_0 \zeta_1, \bar \alpha_0 \bar \zeta_1\rangle$ and $a$ and $b$ are two nonnegative real numbers.
We have 
\[
\frac{||u_{k+1}||^2}{||u_{k}||^2} = \phi(k) + O\Big((\lambda_3/\lambda_1)^k \Big) \;.
\]
Extending $\phi$ to the set of real numbers, we can see that it is a periodic function with period $\frac{\pi}{\theta_1}$. Moreover
\begin{align*}
%\phi(x) &= |\lambda_1|^2 \frac{a + b \cos(2 x \theta_1 + \phi) \cos(2 \theta_1) - b \sin(2 x \theta_1 + \phi) \sin(2\theta_1)}{a + b \cos(2 x \theta_1 + \phi) } \\
%&= |\lambda_1|^2 \Big(1 - b(1-\cos(2 \theta_1))\frac{\cos(2 x \theta_1 + \phi)}{a + b \cos(2 x \theta_1 + \phi)} - b\sin(2\theta_1)\frac{\sin(2 x \theta_1 + \phi)}{a + b \cos(2 x \theta_1 + \phi)}\Big) \;.
\phi(x) &= |\lambda_1|^2 \Big(1 + b\frac{\cos(2 (x+1) \theta_1 + \varphi) - \cos(2 x \theta_1 + \varphi)}{a + b \cos(2 x \theta_1 + \varphi)} \Big) \\
&= |\lambda_1|^2 \Big(1 - 2 b\frac{\sin((2 x+1) \theta_1 + \varphi)\sin(\theta_1)}{a + b \cos(2 x \theta_1 + \varphi)} \Big) \;.
\end{align*}
where we used the formula $\cos(a) - \cos(b) = - 2 \sin(\frac{a+b}{2})\sin(\frac{a-b}{2})$.
%We now make the following simplification, using the fact that $\theta_1$ is usually small:
%\begin{align*}
%\phi(x) = |\lambda_1|^2 \Big(1 - 2 b \theta_1 \frac{\sin(2 x \theta_1 + \phi)}{a + b \cos(2 x \theta_1 + \phi)}\Big) + O(\theta_1^2) \;.
%\end{align*}
Differentiating $\phi$ with respect to $x$, we get:
\begin{align*}
\phi'(x) &= \frac{-2b\sin(\theta_1)|\lambda_1|^2}{(a + b \cos(2 x \theta_1 + \varphi))^2}\Big(2\theta_1 \cos((2x+1)\theta_1 + \varphi) (a+b\cos(2x\theta_1+\varphi)) \\
& \qquad \qquad+ 2b\theta_1 \sin((2x+1)\theta_1+\varphi) \sin(2x\theta_1+\varphi)\Big) \\
& = \frac{-4b\theta_1\sin(\theta_1)|\lambda_1|^2}{(a + b \cos(2 x \theta_1 + \varphi))^2}\Big(a \cos((2x+1)\theta_1 + \varphi) + b \cos(\theta_1)\Big)
\end{align*}
We can see that the maximum and minimum are attained when $\cos(2x\theta_1 + \theta_1+ \varphi) = -\frac {b \cos(\theta_1)}{a }$. 
%Hence
%$\cos(2x\theta_1 + \phi) = \cos((2x+1)\theta_1 + \phi)\cos(\theta_1) + \sin((2x+1)\theta_1 + \phi) \sin(\theta_1) = -b/a \cos(\theta_1)^2 + \sqrt{1-b^2/a^2\cos(\theta_1)^2}\sin(\theta_1)$ and, assuming without loss of generality $\theta_1 \geq 0$,
%\begin{align*}
%\min_x \phi(x) &= |\lambda_1|^2\Big( 1 - 2 b \sin(\theta_1) \frac{\sqrt{1-b^2/a^2\cos(\theta_1)^2}}{a - b^2/a\cos(\theta_1)^2+b\sqrt{1-b^2/a^2\cos(\theta_1)^2}\sin(\theta_1)}\Big) \\
%&= |\lambda_1|^2\Big( 1 -  \frac{2 b/a \sin(\theta_1)}{\sqrt{1-b^2/a^2\cos(\theta_1)^2} + b/a \sin(\theta_1)}\Big) \;.
%\end{align*}
%Similarly,
%\begin{align*}
%\max_x \phi(x) &= |\lambda_1|^2\Big( 1 +\frac{2 b/a \sin(\theta_1)}{\sqrt{1-b^2/a^2\cos(\theta_1)^2} - b/a \sin(\theta_1)}\Big) \;.
%\end{align*}
Let $\underline{x}$ be a minimizer and $\bar x$ a maximizer of $\phi$. 
These are two different numbers such that $\cos(2\underline x\theta_1 + \theta_1+ \varphi)=\cos(2\bar x\theta_1 + \theta_1+ \varphi)$.
Hence, $x^* = \frac{\underline{x}+ \bar x}{2}$ satisfies $\sin(2x\theta_1 + \theta_1+ \varphi) = 0$ and so $\phi(x^*) = |\lambda_1|^2$.
\end{proof}

If we denote $u_k = z_{k+1} - z_k$, since $u_{k+1} = R u_k$ with $R$ defined in~\eqref{eq:Rmatrix}, we can apply the proposition to PDHG and get the following algorithm to estimate for $|\lambda_1|$ with precision $\hat{|\lambda_1|} = |\lambda_1| + O(\min(\theta_1+(\lambda_3/\lambda_1)^{k}, (\lambda_3/\lambda_1)^{k-\pi/\theta_1}))$:
\begin{algorithm}
\begin{algorithmic}
\State Set $\delta = 0.6$, $\epsilon_1 = 10^{-3}$ and $\epsilon_2 = 10^{-5}$
\State Monitor $r_k = \frac{ \|z_{k+1} - z_{k}\|_{V_s}}{ \|z_{k} - z_{k-1}\|_{V_s}}$ \If{$\|z_{k+1} - z_{k}\|_{V_s} \leq \delta \|z_{s} - z_{s-1}\|_{V_{s-1}}$ and $|(1-r_{k+1})/(1-r_k)  -1| \leq \epsilon_1$ and $|r_{k+1}-2r_k + r_{k-1}|/(1-r_k)^2 \leq \epsilon_2$}
\State Return $|\hat\lambda_1| = r_k = \frac{||z_{k+1}-z_k||_{V_s}}{||z_k-z_{k-1}||_{V_s}}$
\EndIf
\If{$\|z_{k+1} - z_{k}\|_{V_s} \leq \delta \|z_{s} - z_{s-1}\|_{V_{s-1}}$ and a local minimum and a local maximum are visible, respectively at $\underline k$ and $\bar k$}
\State Return $\hat{|\lambda_1|} = r_{k^*}$, where $k^* = \lceil (\underline k + \bar k) / 2 \rceil$.
\Else
\State Continue
\EndIf
\end{algorithmic}
\caption{Algorithm to estimate $|\lambda_1|$}
\label{alg:residual_norm_cycles}
\end{algorithm}

We can make the estimation a bit more precise by fitting a quadratic around the local extrema and a linear function around the point where we think $\sin(2x\theta_1 + \theta_1+ \varphi) = 0$.
We get a precision $O(\theta_1 + (\lambda_3/\lambda_1)^{k})$ if no period has been fully observed (case $\theta_1 = 0$ or $\theta_1 \ll 1$) and we get a precision $O((\lambda_3/\lambda_1)^{k - \pi/\theta_1})$ when we have observed a full period.

\begin{figure}
	\centering
\includegraphics[width=0.49\linewidth]{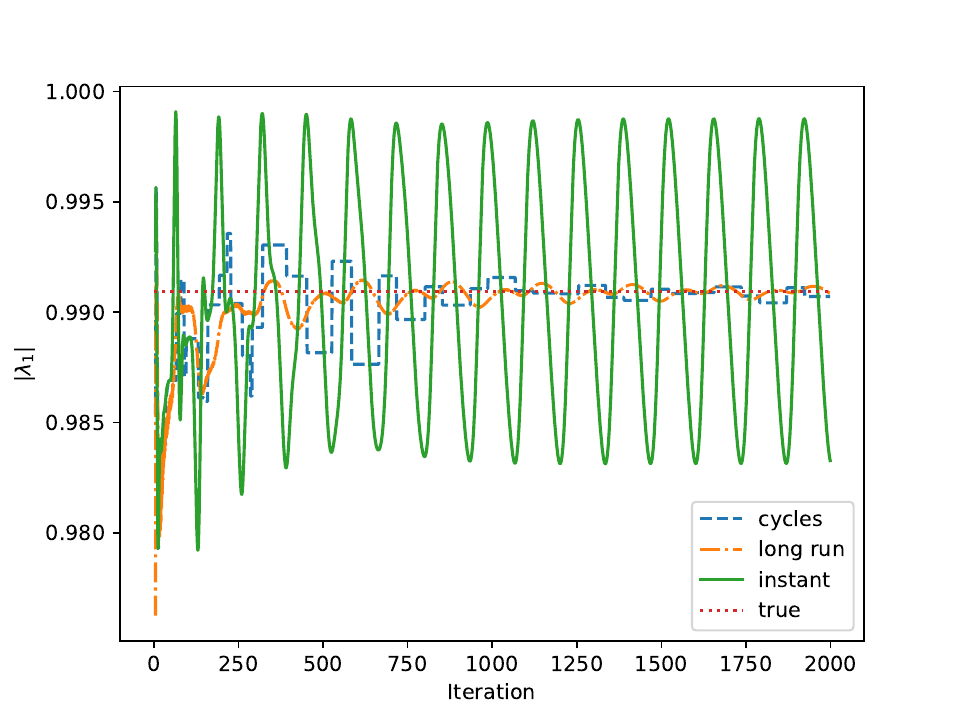}
\includegraphics[width=0.49\linewidth]{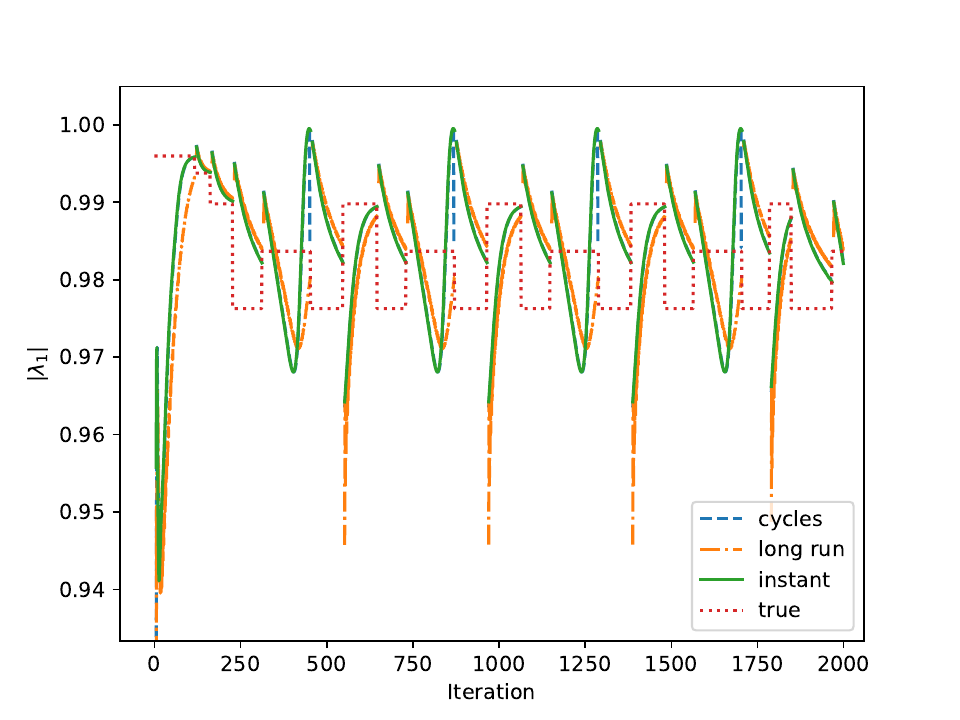}
\caption{Comparison of several estimates of the spectral radius. The true value is in dotted red line, the instantaneous estimate $||u_{k+1}||_V/||u_k||_V$ is in solid green line, the long run estimate $(||u_{k}||_V/||u_{k-s}||_V)^{1/s}$ is in orange dash-dotted line and the estimate proposed in this paper based on the study of cycles is in blue dashed line. The problem under consideration is the same as for Figure~\ref{fig:compareVand2norms}. On the left plot, the step sizes are constant with $\tau = 10 / \|A\|$. The instantaneous estimate fails because it oscillates. One can remark that the oscillations are far from being negligible. The long run and cycle based estimates behave similarly in this context. On the right plot, $\tau$ is modified online by monitoring the convergence rate, starting from $\tau_0 = 100 / \|A\|$, as will be explained in Section~\ref{sec:adap_fercoq}.}
\end{figure}
\begin{figure}	\includegraphics[width=0.5\linewidth]{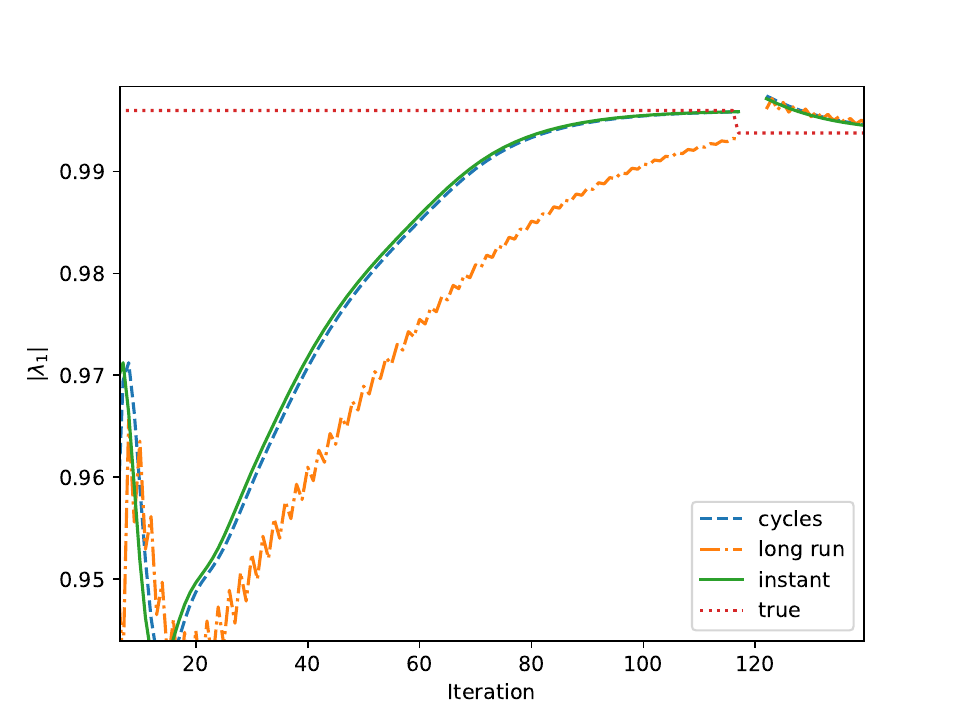}
\includegraphics[width=0.5\linewidth]{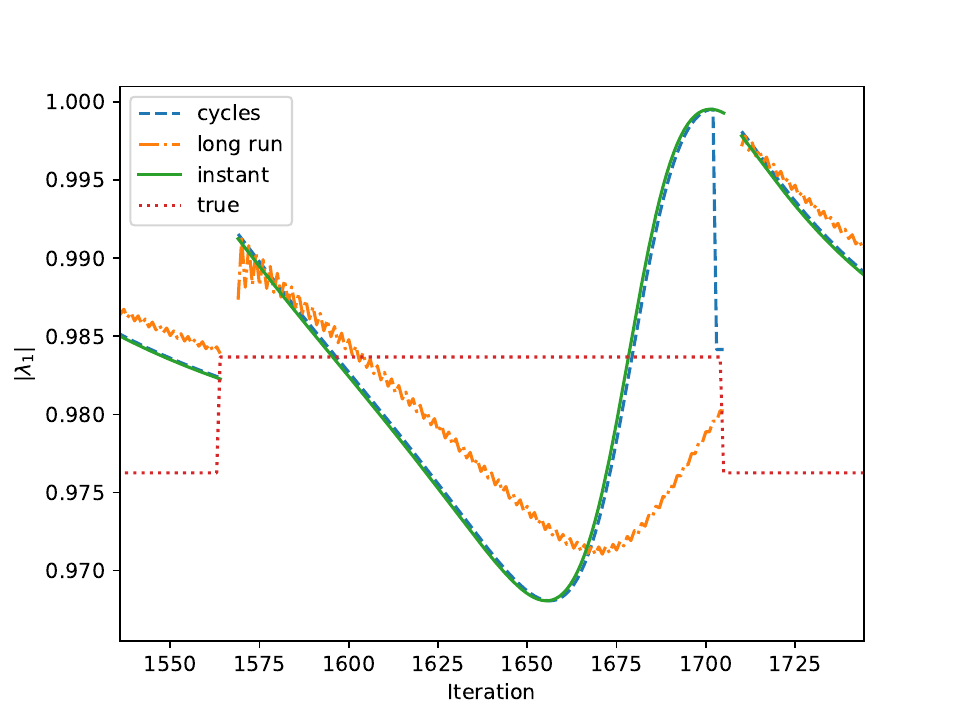}
\caption{Zooms on two cases for rate estimation. 
When the matrix is changing, we may encounter either the case where the principal eigenvalue is unique (left plot) or the case where there is a conjugate pain of principal eigenvalues. In both cases, the cycle based estimate is the most accurate, and is able to take profit of warm start when we modify the step sizes. Hence one single cycle is enough to get a precision allowing us to discriminate between a better and a worse rate.}
\end{figure}

\subsection{Goldstein warm-up}

If the step sizes are very far from the optimum, estimating the rate is very expensive because the algorithm is very slow.
In the other hand, balancing the residuals can be done at each iteration without waiting for glimpses of convergence.
Hence, we propose to combine Goldstein's adaptive steps~\cite{goldstein2015adaptive} with our step size adaptation based on rate estimation.
Moreover, we can use the automatic stopping of Goldstein's adaptive steps, implemented by the geometric decrease of $\alpha_k$, to switch to rate estimation only when primal-dual oscillations are encountered.

\subsection{The adaptive algorithm}

\label{sec:adap_fercoq}

Our proposed adaptive step size algorithm for PDHG is given in Algorithm~\ref{alg:adap_fercoq}.

\begin{algorithm}
	\begin{algorithmic}
		\State $z_0 \in \mathbb R^{n+m}$, $\tau_0, \sigma_0$ such that $\tau_0 \sigma_0 < \|A\|^2$, $u_0 = 1$, $\rho_{\text{est}}^0 = 1$, $\delta = 0.6$, $r = 1.5$, $\alpha_0 = 0.5$, $l=0$, $s_{-1} = s_0=0$
		\For{$k \in \mathbb N$}
		\State $z_{k+1} = \mathrm{PDHG}(z_k, \tau_s, \sigma_s)$  \hfill {\em Run one iteration of PDHG}
		\State $(\bar \tau, \bar \sigma, \bar \alpha) = \textrm{Goldstein}(z_k, z_{k+1}, \tau_s, \sigma_s, \alpha_s)$  \hfill {\em Try Goldstein et al.'s step size adaptation}
		\If{$\bar \tau \neq \tau_s$}
		\State $s = k+1$, $\alpha_{s+1} = \bar \alpha$   \hfill {\em When $\alpha_s$ becomes too small, we skip Goldstein et al.'s rule}
		\State $\tau_s = \bar \tau$, $\sigma_s = \bar \sigma$ 
		\EndIf
		\State Compute $|\hat \lambda_1|(k)$ using Algorithm~\ref{alg:residual_norm_cycles} on iterates $\{s, \ldots, k+1\}$\hfill  {\em Estimate convergence rate}
		\If{Algorithm~\ref{alg:residual_norm_cycles} has returned a value}
		\If{$|\hat \lambda_1|(k) < |\hat \lambda_1|(s)$}
		\State $u_{k+1} = -u_{s}$  \hfill  {\em Revert gear}
		\EndIf
		\State $\tau_{k+1} = \tau_{s} r^{u_{k+1}}$  \hfill  {\em Update step sizes}
		\State $\sigma_{k+1} = \sigma_{s} r^{-{u_{k+1}}}$
		\State $s = k+1$
		\EndIf
		\EndFor
	\end{algorithmic}
	\caption{Adaptive stepsizes for PDHG based on rate estimate with Goldstein warm-up}
	\label{alg:adap_fercoq}
\end{algorithm}

\begin{figure}
\includegraphics[width=0.49\linewidth]{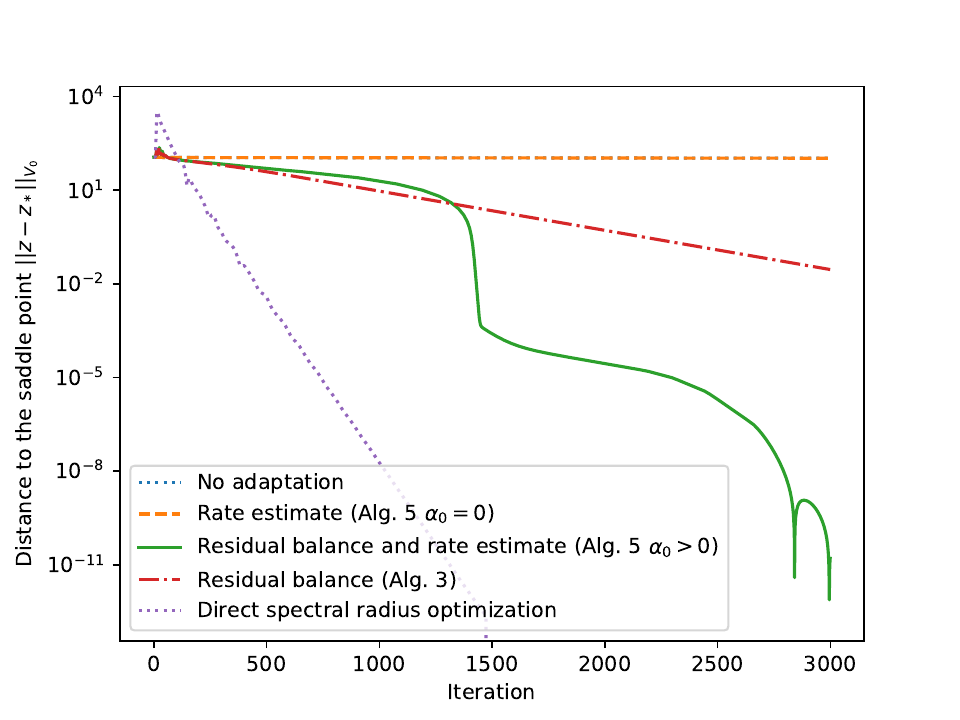}
\includegraphics[width=0.49\linewidth]{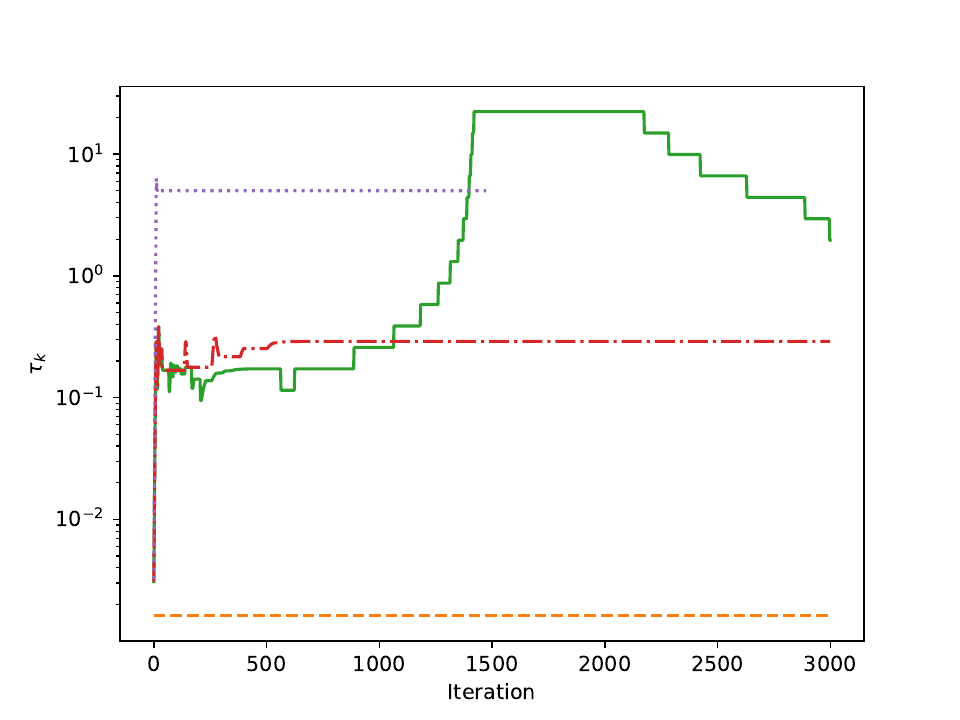}
\caption{Comparison of adaptive step algorithms on the toy quadratic problem of Figure \ref{fig:compareVand2norms}, where we initialized $\tau_0 = \frac{0.01}{\|A\|}$. Left: distance to the saddle point. Right: value of $\tau_k$ for each iteration (same line colors). We had chosen an initial step size value which is far from the optimal one. Thus the base algorithm is quite slow. Moreover, when trying to estimate the rate, we need many iteration before being able to discriminate between two slow rates. Residual balance yields a very quick update of the step sizes to something reasonable. We can see an actual decrease on the left plot. However, the algorithm is not able to deal with the oscillating behavior of the residuals: it quickly drains its updating budget and stalls. Combining both methods gives the solid green line. Residual balance gives a sufficiently good step size allow accurate rate estimates. After a few updates, we obtain a rate nearly as good as what can be obtained if we directly optimize the spectral radius using 0th-order optimization.}
\label{fig:toy_problem}
\end{figure}

In order to prove the convergence of the adaptive algorithm, we shall use the following result.
\begin{lemma}[Lemma~1 in \cite{tran2020adaptive}]
Suppose that $g$ is $\hat M_g$-Lipschitz on its domain (which means it is the sum of a Lipschitz continuous function and the indicator of a convex set). Let $z_* = (x_*, y_*)$ be a saddle point of \eqref{pb:saddle_point}. Define the smoothed gap as 
\begin{multline}
G_{\beta}(z, \dot z) = \max_{x' \in \mathbb R^n, y' \in \mathbb R^m} f(x) + f_2(x) + \langle Ax, y'\rangle - g^*(y') - \frac{\beta_x}{2}\|y' - \dot y\|^2 \\- f(x') - f_2(x') - \langle Ax', y\rangle + g^*(y) - \frac{\beta_y}{2}\|x'-\dot x\|^2
\label{eq:smoothed_gap}
\end{multline}
and $S_\beta(x, \dot y) = G_{(\beta, \beta)}((x, y_*), (x_*, \dot y))$.
Define also $w$ as the projection of $Ax$ on $\dom g$ and $P^* = f(x_*) + f_2(x_*) + g(Ax_*)$. Then we have
\begin{align*}
&f(x) + f_2(x) + g(w) - P^* \geq - \|y_*\| \dist(Ax, \dom g^*) \\
&f(x) + f_2(x) + g(w) - P^* \leq S_{\beta}(x, \dot y) + \beta(2 \hat M_g + \|\dot y\|) \Big(\|\dot y - y_*\| + \sqrt{\|\dot y - y_*\|^2 + \frac{2}{\beta} S_\beta(x, \dot y)} \Big) \\
& \dist(Ax, \dom g) \leq \beta  \Big(\|\dot y - y_*\| + \sqrt{\|\dot y - y_*\|^2 + \frac{2}{\beta} S_\beta(x, \dot y)} \Big)
\end{align*}
\label{lem:smoothed_gap_and_optimality}
\end{lemma}
This lemma shows that if $\dot y$ is close to $y_*$ and $ S_\beta(x, \dot y)$ is small, then the feasibility and optimality gaps are both small.
We shall now prove the convergence theorem for Algorithm \ref{alg:adap_fercoq}. We will do the proof using Algorithm~\ref{alg:tri-pd} version of PDHG but one can also use similar arguments for Algorithm~\ref{alg:vu-condat}.

\begin{theorem}
Let $(z_k)$ be the sequence generated by Algorithm \ref{alg:adap_fercoq} with $\alpha_0 = 0$.
If we update the residuals when the squared residuals have decreased by at least a factor $\delta$ and multiply them by a factor $r$ and $1/r$ respectively where $r > 1$ and $\delta r < 1$, then the feasibility and optimality gaps defined in Lemma~\ref{lem:smoothed_gap_and_optimality} converge to 0.
\label{thm:convergence}
\end{theorem}
\begin{proof}
Let $k \in \mathbb N$ and $s_1, \ldots, s_L$ the iterations where a change in step sizes has occurred, such that $s_L < k \leq s_{L+1}$. We will start by upper bounding $\|z_{k+1} - z_k\|^2_{V_{S_L}}$.

By construction in Algorithm~\ref{alg:residual_norm_cycles}, for all $l$, 
\begin{equation*}
\|z_{s_l+1} - z_{s_{l}}\|^2_{V_{s_l}} \leq \delta 
\|z_{s_{l-1}+1} - z_{s_{l-1}}\|^2_{V_{s_{l-1}}} \;.
\end{equation*}
 Hence,
\begin{equation*}
\|z_{s_L+1} - z_{s_{L}}\|^2_{V_{s_L}} \leq \delta^L 
\|z_{1} - z_{0}\|^2_{V_{0}} \;.
\end{equation*}

Moreover, by firm nonexpansiveness of the PDHG operator~\cite{fercoq2022quadratic}, we know that there exists $\lambda > 0$ such that for all $j \in s_L, \ldots, k$ and $z_* \in \mathcal Z_*$,
\begin{align*}
\lambda \|z_{j+1} - z_j\|_{V_{s_L}}^2 + \|z_{j+1} - z_*\|_{V_{s_L}}^2 \leq \|z_{j} - z_*\|_{V_{s_L}}^2
\end{align*}
and for all $j \in s_L+1, \ldots, k$,
\begin{align*}
\lambda \|z_{j+1} - 2 z_j + z_{j-1}\|^2_{V_{s_L}} + \|z_{j+1} - z_j\|_{V_{s_L}}^2 \leq \|z_{j} - z_{j-1}\|_{V_{s_L}}^2  \;.
\end{align*}
We deduce from this that $\|z_{k} - z_{k-1}\|^2_{V_{s_L}} \leq \|z_{s_L+1} - z_{s_{L}}\|^2_{V_{s_L}}$ and 
\begin{align}
(k - s_L)\|z_{k} - z_{k-1}\|^2_{V_{s_L}} &\leq \sum_{j=s_L}^{k-1} \|z_{j+1} - z_j\|_{V_{s_L}}^2 \leq \frac 1 \lambda \Big( \|z_{s_L} - z_*\|_{V_{s_L}}^2 - \|z_{k} - z_*\|_{V_{s_L}}^2 \Big) \label{eq:bounding_diff_iterates_k}
\end{align}
Similarly, 
\begin{align}
(s_{l+1} - s_l)\|z_{k} - z_{k-1}&\|^2_{V_{s_L}} \leq 
(s_{l+1} - s_l)\|z_{s_L+1} - z_{s_L}\|^2_{V_{s_L}}
\leq (s_{l+1} - s_l) \delta \|z_{s_{L-1}+1} - z_{s_{L-1}}\|^2_{V_{s_{L-1}}} \notag \\
&\leq (s_{l+1} - s_l) \delta^{L-l-1} \|z_{s_{l+1}+1} - z_{s_{l+1}}\|^2_{V_{s_{l+1}}} \notag\\
& \leq (s_{l+1} - s_l) r\delta^{L-l-1} \|z_{s_{l+1}+1} - z_{s_{l+1}}\|^2_{V_{s_{l}}} \leq (s_{l+1} - s_l) r^{L-l} \|z_{s_{l+1}+1} - z_{s_{l+1}}\|^2_{V_{s_{l}}}   \notag \\
&\leq 
r^{L-l} \sum_{j=s_l}^{s_{l+1}-1} \|z_{j+1} - z_j\|_{V_{s_l}}^2 \leq \frac {r^{L-l}}{\lambda} \Big( \|z_{s_l} - z_*\|_{V_{s_l}}^2 - \|z_{s_{l+1}} - z_*\|_{V_{s_l}}^2 \Big) \notag \\
& \leq  \frac {r^{L-l}}{\lambda} \|z_{s_l} - z_*\|_{V_{s_l}}^2 - \frac {r^{L-l-1}}{\lambda} \|z_{s_{l+1}} - z_*\|_{V_{s_{l+1}}}^2 \;. \label{eq:bounding_diff_iterates}
\end{align}

%Yet, we have 
%$\|z_{s_l} - z_*\|_{V_{s_l}}^2 \leq r \|z_{s_l} - z_*\|_{V_{s_{l-1}}}^2$, and 
%$\|z_{s_l} - z_*\|_{V_{s_{l-1}}}^2 \leq \|z_{s_{l-1}} - z_*\|_{V_{s_{l-1}}}^2$ so that $\|z_{s_l} - z_*\|_{V_{s_L}}^2 \leq r^{l-1} \|z_{0} - z_*\|_{V_{0}}^2$.

Summing \eqref{eq:bounding_diff_iterates} for $0 \leq l \leq L-1$ with \eqref{eq:bounding_diff_iterates_k} and using the fact that updating the step sizes changes the norm $\|\cdot\|_{V_s}$ by a factor at most $\sqrt r$, we get
\begin{align*}
&k \|z_{k} - z_{k-1}\|_{V_{s_L}}^2 \leq \frac{r^{L}}{\lambda}\|z_0 - z_*\|^2_{V_0} - \frac{1}{\lambda} \|z_k - z_*\|^2_{V_{s_L}} \\
&\|z_{k} - z_{k-1}\|_{V_{s_L}}^2 \leq \frac{r^{L}}{\lambda k}\|z_0 - z_*\|^2_{V_0}\;.
\end{align*}
Combining both upper bounds on $\|z_{k} - z_{k-1}\|_{V_{s_L}}^2$, we get
\begin{equation}
\|z_{k} - z_{k-1}\|_{V_{s_L}}^2 \leq \min\Big(\delta^L \|z_1 - z_0\|^2_{V_0},  \frac{r^{L}}{\lambda k}\|z_0 - z_*\|^2_{V_0}\Big) \; .
\label{bound_on_residuals}
\end{equation}

Following \cite[Theorem 1]{fercoq2022quadratic}, for all $\beta >0$, 
\begin{align*}
G_{\beta/\tau, \beta/\sigma}(\bar z_{k+1}, z_*) &\leq \frac{1+a_2^+}{2} \|z_{k+1} - z_*\|^2_{V_k} - \frac{1+a_2^+}{2} \|z_{k} - z_*\|^2_{V_k} + \frac{1}{\beta}\|z_{k+1} - z_k\|^2 \\
& \leq \frac{1+a_2^+}{2} \Big(\langle z_{k+1} - z_{k}, z_{k+1} + z_{k} - 2 z_*\rangle_{V_k}\Big) + \frac{1}{\beta}\|z_{k+1} - z_k\|^2 \\
& \leq (1+a_2^+) \|z_{k+1} - z_k\|_{V_k} \|(z_{k+1} + z_k) / 2 - z_*\|_{V_k} + \frac{1}{\beta}\|z_{k+1} - z_k\|^2 \\
& \leq (1+a_2^+) \|z_{k+1} - z_k\|_{V_k} \|z_k - z_*\|_{V_k} + \frac{\max(\tau_k, \sigma_k)}{\beta}\|z_{k+1} - z_k\|^2_{V_k}
\end{align*}
Note that this proof technique is reminiscent to \cite{davis2016convergence}.
Inserting \eqref{bound_on_residuals}, we obtain
\begin{align*}
G_{\beta/\tau, \beta/\sigma}(\bar z_{k+1}, z_*) &\leq
(1+a_2^+) \|z_0 - z_*\|_{V_0}\min\Big((\delta r)^{L/2} \|z_1 - z_0\|_{V_0},  \frac{ r^{L}}{\sqrt{\lambda k}}\|z_0 - z_*\|_{V_0}\Big) \\
& \qquad\qquad+ \frac{\max(\tau_0, \sigma_0)}{\beta} \min\Big((\delta r)^L  \|z_1 - z_0\|^2_{V_0},  \frac{r^{2L}}{\lambda k}\|z_0 - z_*\|^2_{V_0}\Big)
\end{align*}

In order to get a bound independent of $L$, we consider its maximum with respect to $L$. Let $\varphi$ defined by
\begin{equation*}
\varphi(x) = \min\Big(a b^x , c d^x\Big)
\end{equation*}
where $a = \|z_1 - z_0\|_{V_0}^2$, $b = \delta r$, $c = \frac{\|z_0 - z_*\|^2_{V_0}}{\lambda k}$ and $d=r^2$. We have that $b<1$ and $d>1$ and
\[
G_{\beta/\tau, \beta/\sigma}(\bar z_{k+1}, z_*) \leq (1+a_2^+) \|z_0 - z_*\|_{V_0}
\sqrt{\max_{x\geq 0} \varphi(x)} + \frac{\max(\tau_0, \sigma_0)}{\beta} \max_{x\geq 0} \varphi(x)
\]
Now, $(x \mapsto a b^x)$ is decreasing and $(x \mapsto c d^x)$ is increasing. Hence, if $a \leq c$, then $a b^x \leq c d^x$ for all $x \geq 0$ and the maximum of $\varphi$ is attained for $x = 0$ with value $\max_x \varphi(x) = \varphi(0) = a$. Note that since $\lim_{k \to +\infty} c = 0$, this first case will happen only in the beginning of the run. Otherwise, $a > c$ and the maximum of $\varphi$ is attained when $a b^x = c d^x$, that is for $x = \frac{\log(a / c)}{\log(d/b)}$.
In this second case,
\begin{align*}
\max_{x \geq 0} \varphi(x) &= a \exp\Big(\frac{\log(a / c)}{\log(d/b)} \log b\Big)\\
&= \exp\Big(\frac{\log(a) \log(b) - \log(c) \log(b) + \log(a) \log(d) - \log(a) \log(b)}{\log(d) - \log(b)} \Big) \\
& = \exp\Big(\frac{\log(a) \log(d)- \log(c) \log(b)}{\log(d) - \log(b)} \Big) \\
& = \exp\Big(\frac{\log(\|z_1 - z_0\|_{V_0}^2)\log(r^2)}{\log(r/\delta)}\Big)\exp\Big(-\frac{\log(k \lambda / \|z_0 - z_*\|^2_{V_0}) \log(1/(\delta r))}{\log(r/\delta)} \Big)
\\ 
& = \|z_1 - z_0\|_{V_0}^{\frac{4\log(r)}{\log(r/\delta)}}
\Big(\frac{k \lambda}{\|z_0 - z_*\|^2_{V_0}}\Big)^{- \frac{\log(1/(\delta r))}{\log(r/\delta)}}
\end{align*}
Since $ \frac{\log(1/(\delta r))}{\log(r/\delta)} > 0$, this ensures that $\lim_{k \to +\infty} \max_{x\geq 0} \varphi(x) = 0$ and thus for any $\beta > 0$, 
\[
\lim_{k \to +\infty} G_{\beta/\tau, \beta/\sigma}(\bar z_{k+1}, z_*) = 0 \;.
\]
We now use the fact that $G_{\beta/\tau, \beta/\sigma}(\bar z_{k+1}, z_*) \geq S_{\beta}(\bar x_{k+1}, y_*)$ \cite[Prop. 8]{fercoq2022quadratic} and $\dot y = y_*$ in Lemma~\ref{lem:smoothed_gap_and_optimality} to get
\begin{align*}
\dist(A\bar x_{k+1}, \dom g) \leq \sqrt{2 \beta G_{\beta/\tau, \beta/\sigma}(\bar z_{k+1}, z_*) } \underset{k \to +\infty}{\longrightarrow} 0 \;.
\end{align*}
Similarly,
\[
f(\bar x_{k+1}) + f_2(\bar x_{k+1}) + g(\proj_{\dom_g}(A \bar x_{k+1})) \underset{k \to +\infty}{\longrightarrow} 0 \;. \qedhere
\]
\end{proof}

\begin{remark}
The proof gives a sublinear convergence speed but it is in fact quite pessimistic.
Like the convergence proof of \cite{goldstein2015adaptive}, the theorem only says that even if the algorithm chooses completely crazy step sizes, the safe guard $\delta r^2 < 1$ will ensure that convergence still holds. Yet, the novelty of Theorem~\ref{thm:convergence} is that we allow the step sizes to be updated even in an asymptotic regime whereas Goldstein et al.\ forced an arbitrary slow down that implies that step size adaption only occurs in the first iterations of the algorithm.
\end{remark}

\section{More general convex-concave saddle point problems}

In~\cite{liang2018local}, Liang, Fadili and Peyré showed that PDHG enjoys finite-time activity identification under a natural non-degeneracy condition.
Their result is proved for constant step sizes but I conjecture that it holds for varying step sizes as soon as the condition of Theorem~\ref{thm:convergence} is satisfied. In particular, they showed that there exists a matrix $R$ and a saddle point $z_*$ such that 
\begin{equation*}
z_{k+1} - z_* = R (z_k - z_*) + o(z_k-z_*) \;.
\end{equation*}
Moreover, the matrix $R$ corresponds to an averaged operator.
Hence, the analysis in the quadratic case describes the behavior of PDHG when activity identification has taken place and the iterates are close enough to a saddle point.
Since Algorithm~\ref{alg:residual_norm_cycles} uses only residual norms to estimate the rate, we can run it even if the problem is not quadratic and it will eventually give sensible results.

Finally, it is shown in Section 4.1 of~\cite{liang2018local} that for linear programs, the asymptotic rate depends only on the product $\sigma \tau$. Hence none of the methods presented in this paper will have an influence on the convergence rate when solving linear programs with PDHG.

\section{PURE-CD}

In order to solve problems in large dimensions, a very efficient technique is to use coordinate update versions of PDHG. A generalization of Goldstein et al's update rule for S-PDHG has been proposed in~\cite{chambolle2023stochastic}.
In this section, we propose a residual balance and a convergence monitoring solution for PURE-CD~\cite{alacaoglu2020random}. The advantage of this algorithm when compared to S-PDHG is its ability to leverage sparsity in the matrix $A$ when choosing which primal and dual variables are going to be updated.

The algorithm is based on the dual version of Algorithm \ref{alg:tri-pd}. At each iteration, a primal coordinate $i_{k+1}$ is selected at random together with all the dual coordinates $j$ such that $A_{i_{k+1}, j} \neq 0$. We shall denote $J(i) = \{j \in \{1, \ldots, m\} \;:\; A_{i,j} \neq 0\}$  and $I(j) = \{i \in \{1, \ldots, n\} \;:\; A_{i,j} \neq 0\}$. We obtain Algorithm~\ref{alg:pure-cd} below.

\begin{algorithm}
	\caption{Primal-dual method with random extrapolation and coordinate descent (PURE-CD)}
	\label{alg:pure-cd}
	\begin{algorithmic}
		\State { \textbf{Input:}} Diagonal matrices $\theta, \tau_k, \sigma_k > 0$, chosen according to~\eqref{eq: theta_choice}.
		\For{$k = 0,1\ldots $} 
		\State $\bar{y}_{k+1} = \prox_{\sigma_k, g^\ast}\left(y_k + \sigma_k Ax_k\right)$
		\State $\bar{x}_{k+1} = \prox_{\tau_k, f}\left(x_k - \tau_k \left(\nabla f_2(x_k) + A^\top \bar{y}_{k+1}\right)\right)$
		\State Draw $i_{k+1} \in \{ 1, \dots, n \}$ with $\mathbb{P}(i_{k+1} = i) = p_i$
		\State $x_{k+1}^{i_{k+1}} = \bar{x}_{k+1}^{i_{k+1}}$, \quad $x_{k+1}^{i} = x_k^i, \forall i \neq i_{k+1}$
		\State $y_{k+1}^{j} = \bar{y}_{k+1}^j + \sigma^j_k \theta_j(A (x_{k+1} - x_k))_j, \forall j \in J(i_{k+1})$, \quad $y_{k+1}^j = y_k^j, \forall j\not\in J(i_{k+1})$
		\EndFor
	\end{algorithmic}	
\end{algorithm}
The algorithm is guaranteed to converge if the functions are convex and  algorithmic parameters satisfy 
\begin{equation}
\pi_j = \sum_{i \in I(j)} p_i,\quad \underline p = \min_i p_i,\quad \theta_j = \frac{\pi_j}{\underline p}, \quad \sigma_k \equiv \sigma \text{ and } 
\tau_k^i \equiv \tau^i < \frac{2p_i - \underline{p}}{L_i(\nabla f_2) p_i + \underline p^{-1}p_i\sum_{j=1}^m \pi_j \sigma^j A_{j, i}^2} \label{eq: theta_choice}
\end{equation}
where $L_i(\nabla f_2)$ is the $i$th coordinate-wise Lipschitz constant of the derivative of $f_2$. 

In the case $p_i = \frac 1n$ for all $i$ and $f_2 = 0$, \cite{alacaoglu2020random} suggests setting the free parameters $\sigma^j$ as
$\sigma^j = \frac{1}{\theta_j \max_{i'}\|A_{i'}\|}$ and then set $\tau^i = \frac{\gamma \max_{i'} \|A_{i'} \|}{\| A_i \|^2}$,
where $0 < \gamma < 1$. This is working rather well in practice but we shall try to optimize the ratio between primal and dual step sizes by keeping one free parameter $s > 0$, which leads to the step sizes
\begin{align}
&\sigma^j(s) = \frac{s}{\theta_j \max_{i'}\|A_{i'}\|} \label{eq:tau_from_s}\\
&\tau^i(s) = \frac{\gamma (2 - \underline p/p_i)}{L_i(\nabla f_2) + s \|A_i\|^2 / \max_{i'} \|A_{i'}\|} \label{eq:sigma_from_s}
\end{align}

\subsection{Residual balance}

The following result shows how to construct stochastic estimates of the norm of the residuals.
They can be then used to enforce residual balance without the need to compute expensive quantities.

\begin{proposition}
Let us consider the iterates of Algorithm~\ref{alg:pure-cd} and define
\begin{align*}
&d_{k+1} = \sigma_k^{-1} \pi^{-1/2}(y_{k} - y_{k+1}) + \pi^{1/2}(\theta-1) A (p^{-1} (x_{k+1} - x_k))  \\
&p_{k+1} =\tau_k^{-1} p^{-1/2} (x_{k} - x_{k+1}) + p_{i_{k+1}}^{-1/2}(\nabla_{i_{k+1}} f_2(\bar x_{k+1}) - \nabla_{i_{k+1}} f_2( x_{k}))e_{i_{k+1}}
\end{align*}
Then $\mathbb E[\|d_{k+1} \|_2^2 | z_k] = \|\bar d\|^2_2$ for some 
$\bar d \in \partial g^*(\bar y_{k+1}) - A \bar x_{k+1}$ and 
$\mathbb E[\|p_{k+1}\|_2^2 | z_k] = \|\bar p\|^2_2$ 
for some $\bar p \in \partial f(\bar x_{k+1}) + \nabla f_2(\bar x_{k+1}) + A^\top \bar y_{k+1}$.
\label{prop:estimate_subgradient_purecd}
\end{proposition}
\begin{proof}
Just like for Algorithm~\ref{alg:tri-pd}, but using the notation $ab$ for the element-wise product of two vectors $a$ and $b$, we have
\begin{align*}
& 0 \in \sigma_k \partial g^*(\bar y_{k+1}) + \bar y_{k+1} - y_k - \sigma_k A  x_{k} \\
& 0 \in \tau_k \partial f(\bar x_{k+1}) + \tau_k \nabla f_2(x_k) + \bar x_{k+1} - x_k + \tau_k A^\top \bar y_{k+1}
\end{align*}
Now, by the coordinate selection rule,
\begin{align*}
&\mathbb E[x_{k+1}^i | z_k] = (1-p_i) x_k^i + p_i \bar x_{k+1}^i  \\
& \mathbb E[y_{k+1}^j | z_k ] = (1-\pi_j) y_k^j + \pi_j (\bar y_{k+1}^j + \sigma_k^j \theta_j A(\bar x_{k+1} - x_k)_j)
\end{align*}
Hence,
\begin{align*}
& \bar x_{k+1} - x_k = \mathbb E[p^{-1} (x_{k+1} - x_k) | z_k]\\
& \bar y_{k+1} - y_k + \sigma_k \theta A(\bar x_{k+1} - x_k) = \mathbb E[\pi^{-1}(y_{k+1} - y_k) | z_k] \\
& \bar y_{k+1} - y_k  = \mathbb E[\pi^{-1}(y_{k+1} - y_k) | z_k] - \sigma_k \theta \mathbb E[ A (p^{-1} (x_{k+1} - x_k)) | z_k]
\end{align*}
and we get
\begin{align*}
&\mathbb E[\sigma_k^{-1} \pi^{-1}(y_{k} - y_{k+1}) + (\theta-1) A (p^{-1} (x_{k+1} - x_k)) | z_k] \in \partial g^*(\bar y_{k+1}) - A \bar x_{k+1} \\
&\mathbb E[\tau_k^{-1} p^{-1} (x_{k} - x_{k+1}) + p_{i_{k+1}}^{-1}(\nabla_{i_{k+1}} f_2(\bar x_{k+1}) - \nabla_{i_{k+1}} f_2( x_{k}))e_{i_{k+1}} | z_k] \in \partial f(\bar x_{k+1}) + \nabla f_2(\bar x_{k+1}) + A^\top \bar y_{k+1}
\end{align*}

We then need to compare the norms of the residuals. We shall use the following result:

	Let $\bar d$ be some fixed vector and let $D$ be a random variable such that $D_j = \frac{\bar d_j}{\sqrt{\pi_j}}$ if $j \in J(i_{k+1})$ and $D_j = 0$ if $j \not \in J(i_{k+1})$.
	We have $\mathbb E[\|D\|_2^2] = \|\bar d\|_2^2$.
Indeed, since $\pi_j = \sum_{i \in I(j)} p_i$, we have
	\begin{align*}
	\mathbb E[\|D\|_2^2] = \sum_{i=1}^n p_i \sum_{j \in J(i)}\frac{(\bar d_j)^2}{\pi_j} = \sum_{j=1}^m \sum_{i \in I(j)} \frac{p_i}{\pi_j} (\bar d_j)_2^2=
	\sum_{j=1}^m (\bar d_j)_2^2 = \|\bar d\|^2_2  \;.
	\end{align*}

We then define $\bar d = \mathbb E[\sigma_k^{-1} \pi^{-1}(y_{k} - y_{k+1}) + (\theta-1) A (p^{-1} (x_{k+1} - x_k)) | z_k]$ and use the previous result for $D = d_{k+1}$. Similar arguments can be done for the primal residual.
\end{proof}

In this proposition, we need $\nabla_{i_{k+1}} f_2(\bar x_{k+1})$. If $f_2$ is separable, then this requires only $\bar x_{k+1}^{i_{k+1}}$ which is given by the algorithm. If $f_2$ is quadratic, $\nabla f_2$ is affine and we can safely replace $\bar x_{k+1}$ by $x_k + \mathbb E[p^{-1} (x_{k+1} - x_k) | z_k]$.
In those two important cases, we obtain an unbiased estimator of an element in the sub/super-gradient of the Lagrangian. In the general case, we may need to use approximations of $\nabla_{i_{k+1}} f_2(\bar x_{k+1})$ with some bias.

Using Proposition~\ref{prop:estimate_subgradient_purecd}, we can replace the primal and dual residuals in Algorithm~\ref{alg:goldstein} by their stochastic counterparts and obtain a residual balance step size adaptation method.
Then, the analysis of stochastic adaptive step sizes in~\cite{chambolle2023stochastic}, which relies mainly on the almost sure slow down of the updates, can be adapted to ensure that the algorithm retains convergence even with the residual balance based adaptive step sizes.

\subsection{Convergence monitoring}

Monitoring convergence of a stochastic algorithm, is more tricky than for a deterministic one. Even if the Lagrangian is a quadratic function, the update matrix changes at each iteration. Each of these matrix has no reason to induce a contraction: only their product will convey information on the convergence rate. Moreover, since $z_{k+1} = R_{k+1}z_k + b_{k+1}$, the basic idea of using power iterations on $z_{k+1} - z_k$ breaks down: $z_{k+1} - z_k = R_{k+1}z_k - R_k z_{k-1} + b_{k+1} - b_k$ and this does not amount to observing a matrix product any more.

We thus propose here a less precise technique, but which is compatible with a random algorithm. Let $M : \mathbb R^n \times \mathbb R^m \to \mathbb R_+$ be some optimality measure, that is a computable function such that $M(z) = 0$ if and only if $z$ is a saddle point of the Lagrangian. One may for instance use the self-centered smoothed duality gap as an optimality measure~\cite{fercoq2022quadratic,walwil:hal-04501394}.

Our estimate of the rate will then be constructed from
\[
\hat \rho_{i:j} = \Big(\frac{M(z_{j})}{M(z_i)}\Big)^{1/(j-i)} = \Big(\prod_{k=i}^{j-1} \frac{M(z_{k+1})}{M(z_k)}\Big)^{1/(j-i)} \;.
\]
Our goal is to compare the rates for two different values of the step sizes given by the parameters $s^1$ and $s^2$, where the dependence of the step size on this scalar parameter is given in \eqref{eq:tau_from_s}, \eqref{eq:sigma_from_s}. We shall denote the set of iterations where the step size has been $s^1$ by $K_1 = \{k' \;:\; s_{k'} = s^1\}$ and similarly $K_2$.
Even if it goes in contradiction with what we learnt in Section~\ref{sec:quadratic}, we will assume that for all $k$, $\hat \rho_{k:k+1}$ is an independent identically distributed log-normal random variable with parameters $\log(\rho)$ and $\Sigma^2$.

We thus have $\log(\hat \rho_{K_1}) = \frac{1}{|K_1|} \sum_{l \in K_1} \log(\hat \rho_{l:l+1})$ so that $\log(\hat \rho_{K_1})$ is a normal random variable with mean $\log(\rho_1)$ and variance $\frac{\Sigma_1^2}{|K_1|}$, where $\Sigma_1$ can be estimated as the standard deviation of $\log(\hat \rho_{l:l+1})$, $l \in K_1$. Similarly, we can estimate $\log(\rho_2)$ and $\frac{\Sigma_2^2}{|K_2|}$.
By analogy to \eqref{eq:rate_strong_convex}, we can hope that in favorable cases, 
multiplying $s$ by a factor $\alpha$ or $1/\alpha$ will have an influence on the rate given by 
\begin{align*}
&1 - \rho(\alpha^{-1} s) \!=\! c \min\big(\mu_f \alpha \tau(s), \mu_{g^*}\alpha^{-1} \sigma(s)\big) \in c \min\big(\mu_f \tau(s), \mu_{g^*} \sigma(s)\big) [\alpha^{-1}, \alpha] \!=\! (1 - \rho(s))[\alpha^{-1}, \alpha]  \\
&1 - \rho(\alpha s) = c \min\big(\mu_f \alpha^{-1} \tau(s), \mu_{g^*}\alpha \sigma(s)\big) \in (1 - \rho(s))[\alpha^{-1}, \alpha] \;.
\end{align*}
We would like to be able to discriminate between $\rho(s)$ and $\rho(\alpha s)$ with a sufficiently high probability. 
We have two Gaussian models $R_1$ for the rate at $s=s^1$ and $R_2$ for the rate at $s = s^2$. We thus design a statistical test and choose $s = s^2$ only if we accept the hypothesis that $R_1 > R_2$. This test will be designed using the probability $p = \mathbb P(R_1 > R_2) = 0.5 (1 - \Phi\Big(\frac{\log(\hat \rho_{K_1}) - \log(\hat\rho_{K_2})}{\Sigma_1^2/|K_1| + \Sigma_2^2/|K_2|}\Big))$ where $\Phi$ is Gauss's error function. 
Note that if $\rho_1 \neq \rho_2$, this probability will significantly differ from 0.5 as soon as $K_1$ and $K_2$ are large enough. We present the full procedure in Algorithm~\ref{alg:cm_pure_cd}. $\bar s_k$ is our currently trusted step size, $\underline s_k$ is our current tentative step size and $s_k$ is the currently used step size. If $p < 0.45$, we reject the proposal and try another one. If $p > 0.55$, we accept the proposal and make it our current trusted step size. The value for rejection or acceptation may look quite close to 0.5 but we experienced that this value balances well the trade-off between making no mistake and updating faster the step sizes.

% However, waiting for the variance to decrease enough may force the algorithm to keep the step sizes unchanged. In order to balance the trade-off between avoiding doing errors and allowing fast enough changes in the step sizes, we propose the following randomized procedure.

\begin{algorithm}
	\begin{algorithmic}
	\State $\bar s_{0} > 0$, $r = 2$, $u_0 = 1$, $\underline s_{0} = \bar s_{0} r^{u_0}$ or $\underline s_0$ given by a few iterations of residual balance, $s_0 = \bar s_0$, $0 < \delta < r^{-1}$
	\For{$k \in \mathbb N$}
	\State Run PURE-CD (Algorithm~\ref{alg:pure-cd}) with step sizes set using $s_k$ until $M(z_{k+1}) \leq \delta M(z_{k})$
	\State Set $\bar \rho_k, \bar \Sigma_k$ as the mean and standard deviation of $\bar K = \{ \hat \rho_{l:l+1} \;:\; \sigma_l = \bar \sigma_k\}$
	\State Set $\underline \rho_k, \underline \Sigma_k$ as the mean and standard deviation of $\underline K = \{ \hat \rho_{l:l+1} \;:\; \sigma_l = \underline \sigma_k\}$
	\State Set $p = 1 - \Phi\Big(\dfrac{\log(\bar \rho_k) - \log(\underline \rho_k)}{\bar \Sigma_k^2/|\bar K| + \underline \Sigma_k^2/ |\underline K|} \Big)$
	\If{$p < 0.45$}
	\State $\bar s_{k+1} = \bar s_{k} $
	\State $u_{k+1} = - u_k$  \hfill {\em we revert gear}
	\State $\underline s_{k+1} = \bar s_{k} r^{u_{k+1}}$
	\State $s_{k+1} \in \{\bar s_{k+1}, \underline s_{k+1}\} \setminus \{s_k\}$
	\ElsIf{$0.45 \leq p \leq 0.55$}
	\State $\bar s_{k+1} = \bar s_{k} $, $\underline s_{k+1} = \underline s_{k}$, $u_{k+1} = u_k$
	\State $s_{k+1} \in \{\bar s_{k+1}, \underline s_{k+1}\} \setminus \{s_k\}$ \hfill {\em we try to reduce standard deviation}
	\ElsIf{$p > 0.55$}
	\State $\bar s_{k+1} = \underline s_k$, $u_{k+1} = u_k$
	\State $\underline s_{k+1} = \bar s_{k} r^{u_{k+1}}$ \hfill {\em we proceed further}
	\State $s_{k+1} = \underline s_{k+1}$
	\EndIf
	\EndFor
	\end{algorithmic}
\caption{Convergence monitoring for PURE-CD} \label{alg:cm_pure_cd}
\end{algorithm}

\subsection{AR model for convergence monitoring}
\label{sec:ar}

The model presented in the previous section takes into account the stochastic feature of the algorithm but does not leverage the fact that the iterate sequence may be spiraling to the saddle point, which results in the periodic behavior of naive rate estimates. 

Instead of assuming that $\log(\hat \rho_{l:l+1})$ is an i.i.d. Gaussian process, we may rather assume that $\log(\hat \rho_{l:l+1})$ is an autoregressive process of order 1. This dependence on the past values is motivated by the fact that we may have a pair of conjugate eigenvalues. Indeed, if $u(k) = \sin(\theta k)$, we can write $u(k+1) = \sin(\theta(k+1)) = \cos(\theta) \sin(\theta k) + \sin(\theta) \cos(\theta k) = a_1 u(k) + \epsilon(k+1)$ where $|a_1| \leq 1$ and $|\epsilon(k+1)| \leq \sin(\theta)$.

The stochastic model for the instant rate is then
\[
\log(\hat \rho_{l:l+1}) = a_0(s) + a_1(s) \log(\hat \rho_{l-1:l}) + \Sigma(s) \epsilon_{l+1}
\]
where $a_0(s), a_1(s), \sigma(s)$ are the parameters of the model and $\epsilon_{l+1}$ is an i.i.d. centered standard Gaussian noise. All the parameters will depend on the step size parameter $s$ and the estimate of the rate will be an estimate of $\mathbb E[\log(\hat \rho_{l:l+1})]$, which assuming stationarity of the stochastic process, is given by
\[
\mathbb E[\log(\hat \rho_{l:l+1})] = \frac{a_0(s)}{1 - a_1(s)} \;.
\]
The cross-covariance is given by
\[
r_k(s) = a_1(s) r_{k-1}(s)+ \Sigma(s)^2 \delta_k
\]
so that the variance satisfies %$r_0(s) = \frac{(1-a_2) \Sigma^2}{(1+a_2)(1-a_1-a_2)(1+a_1-a_2)}$ as soon as $|a_2| < 1$, $a_1+a_2 < 1$ and $a_2 - a_1 < 1$.
$r_0(s) = \frac{\Sigma(s)^2}{1-a_1(s)^2}$. We shall estimate the parameters using the least squares problem
\[
(\hat a_0(s), \hat a_1(s)) = \arg\min_{a_0, a_1} \frac12 \sum_{l \in |K_1|} \big(\log(\hat \rho_{l:l+1}) - a_0 - a_1 \log(\hat \rho_{l-1:l}) \big)^2
\]
and define 
\[
\log(\hat \rho(s)) =  \frac{\hat a_0(s)}{1 - \hat a_1(s)} \;.
\]

By \cite{tiao1983consistency}, we know that this estimate is consistent. Its variance will be estimated by
\[
\hat V(s) = \frac{\hat r_0(s)}{|K_1|-1}\frac{1+\hat a_1}{1-\hat a_1}=  \frac{\hat \Sigma(s)^2}{(|K_1|-1)(1-\hat a_1(s))^2} \;.
\]

In Table~\ref{tab:ar_on_toy}, we compare the i.i.d. model with the autoregressive model and the residual balance strategy. We can see that residual balance gives an algorithm which is quite insensitive to the initial step sizes. However, the speed of convergence that we obtain is not optimal. On the other hand, convergence monitoring takes time before finding what good step sizes are, and even more time when the initial guess is far from optimum. We can also see that the autoregressive model looks faster to discriminate rates than the i.i.d. model.
Finally, our combination of residual balance and convergence monitoring is able to check whether step sizes dictated by residual balance are better than the current one, so that the final behavior on this problem is quite promising.

\begin{table}
\begin{tabular}{l|r|r|r|r}
Initial $s$ & 0.001 & 0.1 & 1 & 10 \\
\hline
Constant step size & 114,784 & 7,057 &43,146 & 419,381\\
Residual balance & 37,144 & 37,620 &35,621 & 37,935 \\
Convergence monitoring (i.i.d model) &17,635 &23,958 & 34,287 & 197,316 \\
Convergence monitoring (AR1 model) & 20,562, & 8,233 & 17,119 & 140,005\\
residual balance + i.i.d model &24,371 & 9,482 &15,548 & 14,021 \\	
\end{tabular}
\caption{Number of iterations to reach a duality gap equal to $10^{-10}$ on the toy problem of Figure~\ref{fig:compareVand2norms}. We~performed one single run, so that the figures do not illustrate the stochastic nature of the algorithm. We~can see that combining residual balance and convergence monitoring consistently gives a quick algorithm.}
\label{tab:ar_on_toy}
\end{table}

\subsection{Convergence with varying step sizes}

Let us now show that our adaptive step sizes do not prevent convergence. Like for deterministic PDHG, we have no slow down in the updates, so that the proof of \cite{chambolle2023stochastic} does not apply.
Yet, thanks to Lemma~\ref{lem:smoothed_gap_and_optimality}, we know that we just need to control the smoothed gap defined in \eqref{eq:smoothed_gap} in order to prove convergence of the algorithm.

The convergence proof of PURE-CD \cite{alacaoglu2020random} uses two primal-dual sequences.
The plain sequence that we will denote $(z_{k})$ and an averaged sequence that we will denote $(z_k^{\mathrm{av}})$. 
However, for a technical reason, the averaged sequence is averaging a modification of the plain sequence: $x_k^{\mathrm{av}} = \frac{1}{k}\sum_{k'=1}^k x_{k'}$ but
$y_k^{\mathrm{av}} = \frac{1}{k}\sum_{k'=1}^k \check y_{k'}$ for some vector $\check y_k$.
We now prove a technical lemma that shows that $z_k^{\mathrm{av}}$ does not go too far away from the initial point of the $l$th step, namely $\tilde z_l$.
We will need to use the weighted distance to saddle point introduced in~\cite{alacaoglu2020random} and defined as $\Delta_l(z) = D_p(x_k, z^*_k) + \frac{\underline p}{2}\|x_k - x_k^*\|^2_{\tau_l^{-1}p^{-1}} + \frac{\underline p}{2}\|y_k - y_k^*\|^2_{\sigma_l^{-1}\pi^{-1}}$. For the precise meaning of each term, we refer the reader to \cite{alacaoglu2020random}.
\begin{lemma}
$\exists C>0$ independent of $s_l$ such that for $\gamma = (p^{-1}\tau^{-1}, \pi^{-1}\sigma^{-1}) > 0$,

\noindent $\mathbb E[\|z_k^{\mathrm{av}} - \tilde z_l\|^2_\gamma | \tilde z_l]\leq C \Delta_l(\tilde z_l)$.
\label{lem:pure-cd-averaged-sequence}
\end{lemma}
\begin{proof}
From Lemma 5 in \cite{alacaoglu2020random}, we know that for $\gamma_2 = \pi^{-1}\sigma^{-1} > 0$,
\[
\mathbb E[\|\check y_{k+1} - y_{k+1}\|^2_{\pi^{-1}\sigma^{-1}}| z_k] \leq \|\check y_k - y_k\|^2_{\pi^{-1}\sigma^{-1}} + \|\bar x_{k+1} - x_k\|^2_{B(\pi^{-1}\sigma^{-1})}
\]
where $B(\pi^{-1} \sigma^{-1})_i = p_i \sum_{j=1}^m \theta_j^2 \pi_j^{-1} \sigma_j A_{j,i}^2 \leq c_p \tau^{-1}$ for some $c_p$ that depends only on $p$. Hence, by summing for $k \in \{0, \dots, k\}$ and using the fact that $\check y_0 = y_0$, we get
\[
\mathbb E[\|\check y_{k+1} - y_{k+1}\|^2_{\pi^{-1} \sigma^{-1}} | z_0] \leq c_p \sum_{k'=0}^k \|\bar x_{k'+1} - x_{k'}\|^2_{\tau^{-1}} \leq \frac{2c_p}{C_{\tau, \tilde V}} \Delta_l(\tilde z_l)
\]
where the last inequality follows from (37) in \cite{alacaoglu2020random}. Here 
\begin{equation*}
C_{\tau, \tilde V} = \min_i \tau_i C(\tau)_i = \min_i \frac{2p_i}{\underline p} -1 -p_i \sum_{j=1}^m \pi_j^{-1} \sigma_j \tau_i \theta_j^2 A_{j,i}^2 - \frac{L_i(\nabla f_2) p_i}{\underline p} \tau_i
\end{equation*}
Using $\tau_i = \frac{\gamma (2 - \underline p/p_i)}{L_i(\nabla f_2) + s_l \|A_i\|^2 / \max_{i'} \|A_{i'}\|}$ and $\sigma_j = \frac{s_l}{\theta_j \max_{i'}\|A_{i'}\|}$ we get $C_{\tau, \tilde V}  = (2 \frac{\underline p}{p_i} - 1)(1-\gamma)$,
which is independent of $s$.

Now for the projection of $z_0 = \tilde z_l$ onto the set of saddle points, denoted $z^*_0 = (x^*_0, y^*_0)$,
\begin{align*}
\|z_k^{\mathrm{av}} - \tilde z_l\|^2_\gamma & = \|x_k^{\mathrm{av}} - \tilde x_l\|^2_{p^{-1}\tau^{-1}} + \|y_k^{\mathrm{av}} - \tilde y_l\|^2_{\pi^{-1}\sigma^{-1}} \leq \frac 1k \sum_{k'=1}^k \|x_{k'} - \tilde x_l\|^2_{p^{-1}\tau^{-1}} + \|\check y_{k'} - \tilde y_l\|^2_{\pi^{-1}\sigma^{-1}} \\
& \leq \frac 1k \sum_{k'=1}^k \Big( \|x_{k'} - \tilde x_l\|^2_{p^{-1}\tau^{-1}} + 2 \|\check y_{k'} - y_k\|^2_{\pi^{-1}\sigma^{-1}} + 2\|y_k - \tilde y_l\|^2_{\pi^{-1}\sigma^{-1}}\Big) \\
\mathbb E[\|z_k^{\mathrm{av}} - \tilde z_l\|^2_\gamma |\tilde z_l]& \leq \frac 4k \sum_{k'=1}^k \Big(\|x_{k'} - x_0^*\|^2_{p^{-1}\tau^{-1}} + \|x_0^* - \tilde x_l\|^2_{p^{-1}\tau^{-1}} + \|y_{k'} - y^*_0\|^2_{\pi^{-1}\sigma^{-1}} + \|y_0^* - \tilde y_l\|^2_{\pi^{-1}\sigma^{-1}} \Big) \\
&\qquad \qquad +\frac{4c_p}{C_{\tau, \tilde V}} \Delta_l(\tilde z_l) \\
& \leq \Big(\frac{8}{\underline p} + \frac{4 c_p}{C_{\tau, \tilde V}} \Big)\Delta_l(\tilde z_l)
\end{align*}
where we used the fact that $D_p(x_k, z^*) \geq 0$ and  $D_p(x_k, z^*) + \frac{\underline p}{2}\|x_k - x^*\|^2_{\tau_l^{-1}p^{-1}} + \frac{\underline p}{2}\|y_k - y^*\|^2_{\sigma_l^{-1}\pi^{-1}}$ is decreasing in expectation as a function of $k$ for any saddle point $z^*$ (34) in \cite{alacaoglu2020random}.
\end{proof}

\begin{algorithm}[tbp]
	\caption{PURE-CD with adaptive step sizes}
	\label{alg:pure-cd-adap}
	\begin{algorithmic}
	\State Set $s_0 > 0$, $\epsilon_0 = G_{c_1 s_0, \frac{c_2}{ s_0}}(z_0, z_{0})$
	\For{$l = 0,1\ldots $} 
	\State $k=0$, $z_0 = \tilde z_{l}$
	\While{$G_{\frac{c_1 s_l}{k}, \frac{c_2}{k s_l}}(z^{\mathrm{av}}_{k}, z^{\mathrm{av}}_{k}) > \epsilon_l$}
	\State Generate $z_{k+1}$ and $z^{\mathrm{av}}_{k+1}$ using PURE-CD with step sizes $\tau(s_l)$, $\sigma(s_l)$
	\State $k \leftarrow k+1$
	\EndWhile
	\State Set $k_l=k$ the number of iterations of the inner loop
	\State Set $\tilde z_{l+1} = z_{k_l}$
	\State Set $\epsilon_{l+1} = \delta G_{\frac{c_1}{k_l}, \frac{c_2}{k_l}}(z^{\mathrm{av}}_{k_l}, z^{\mathrm{av}}_{k_l}) $
	\State Choose $s_{l+1}$ such that $r^{-1} \leq \frac{\tau(s_{l+1})}{\tau(s_l)} \leq r$ and $r^{-1} \leq \frac{\sigma(s_{l+1})}{\sigma(s_l)} \leq r$
	\EndFor
	\end{algorithmic}	
\end{algorithm}

\begin{theorem}
Suppose that in Algorithm~\ref{alg:pure-cd-adap}, we choose $\delta < 1$ and multiply or divide by at most $r$ the step sizes. Assume also that $f(x_l)+f_2(x_l)$ and $\sum_{j=1}^m (\pi_j^{-1} - 1)g_j^*(y_l^{(j)})$ are almost surely uniformly bounded as functions of $l$.
Then for any probability level $1-\rho$, we get $G_{\frac{2c_1}{k_l}, \frac{2c_2}{k_l}}(z_{l+1}, z_{l+1}) \leq \delta^L \epsilon_0$ after a number of iterations bounded by
\begin{align*}
\sum_{l=0}^{L-1} k_l \leq \frac{L C_3}{\rho \epsilon_0} \frac{\delta^{-L}-1}{\delta^{-1} - 1}
+ \frac{L C_5 \Delta_0(z_0)}{\rho \epsilon_0} \frac{(r^2/\delta)^L - 1}{r^2/\delta - 1} + L 
\end{align*}
were $c_1$, $c_2$, $C_3$ and $C_5$ are problem-dependent constants.

Moreover, for any $\alpha > 0$, choosing $r = \delta^{-\alpha/2}$ and neglecting probabilistic arguments, we get a precision $\epsilon$ in smoothed gap after
$O(\epsilon^{-\alpha-1})$ iterations.

\end{theorem}
\begin{proof}	
We have $\mathbb E[\Delta_l(z_k) | \tilde z_{l}] \leq \Delta_l(\tilde z_{l})$, from which we get
\begin{align}
\mathbb E[\Delta_l(\tilde z_{l+1}) | z_{l}] = \mathbb E[\Delta_l(z_{k_l}) | z_{l}] \leq \Delta_l(z_{l}) \leq r \Delta_{l-1}(z_{l}) \;. \label{eq:delta_decrease}
\end{align}

For PURE-CD, we can find in \cite{alacaoglu2020random} that there exists constants $c_1 \leq 2 + 2\underline p, c_2\leq 2, C_3, C_4$ that depend on the step sizes only through the product $\tau^i \sigma^j$, and averaged primal-dual iterates $z_k^{\mathrm{av}}$ such that 
\begin{align*}
\mathbb E[G_{\frac{c_1}{k}, \frac{c_2}{k}}(z_k^{\mathrm{av}}, \tilde z_{l}) | \tilde z_l] \leq \frac{C_3 + C_4 \Delta_l(\tilde z_{l})}{k} \;.
\end{align*}
Moreover by \cite[Lemma 6]{fercoq2022quadratic}, 
\[
G_{\frac{c_1}{k}, \frac{c_2}{k}}(z_k^{\mathrm{av}}, \tilde z_{l}) \geq G_{\frac{2c_1}{k}, \frac{2c_2}{k}}(z_k^{\mathrm{av}}, z_k^{\mathrm{av}}) - \frac 1 k\|z_k^{\mathrm{av}}-\tilde z_l\|^2_{(2c_1, 2c_2)} \;.
\]
Using Lemma~\ref{lem:pure-cd-averaged-sequence}, we get
\[
\mathbb E[G_{\frac{2c_1}{k}, \frac{2c_2}{k}}(z_k^{\mathrm{av}}, z_k^{\mathrm{av}})|\tilde z_l] \leq \frac{C_3 + C_4 \Delta_l(\tilde z_{l})}{k} + \max_{i,j}(2c_1 p_i \tau_i, 2c_2 \pi_j \sigma_j) C \frac{\Delta_l(\tilde z_{l})}{k}
\] 
$C_3$ depends on $z_{l}$ but the technical assumption we make ensures that it is uniformly bounded along the run of the algorithm.
We can combine this with \eqref{eq:delta_decrease} to get 
\begin{align*}
\mathbb E[G_{\frac{2c_1}{k}, \frac{2c_2}{k}}(z_k^{\mathrm{av}}, z_k^{\mathrm{av}})] \leq \frac{C_3 + C_5 r^{2l}\Delta_0(z_{0})}{k} \;.
\end{align*}
where $C_5 = C_4 + C \max_{i,j}(2c_1 p_i \tau_i(s_0), 2c_2 \pi_j \sigma_j(s_0))$.

Hence, by Markov's inequality,
\begin{align*}
\mathbb P(G_{\frac{2c_1}{k}, \frac{2c_2}{k}}(z_k^{\mathrm{av}}, z_k^{\mathrm{av}}) > \epsilon_l) \leq \frac{C_3 + C_5 r^{2l} \Delta_0(z_{0})}{k \epsilon_l} \;. 
\end{align*}
This implies that, given $\rho > 0$, as soon as $k_l = \lceil L\frac{C_3 + r^{2l} C_5 \Delta_0(z_{0})}{\rho \epsilon_l}\rceil$, we need to have $\mathbb P(G_{\frac{2c_1}{k_l}, \frac{2c_2}{k_l}}(z_{k_l}^{\mathrm{av}}, z_{k_l}^{\mathrm{av}}) > \epsilon_l) \leq \frac{\rho}{L}$.

Now, we use $\epsilon_l = \delta^l \epsilon_0$ and use a union bound to get that with probability larger than $1-\rho$, 
\begin{align*}
\sum_{l=0}^{L-1} k_l \leq \sum_{l=0}^{L-1} L\frac{C_3 + C_5 r^{2l} \Delta_0(z_{0})}{\rho \delta^l \epsilon_0} + 1 \leq \frac{L C_3}{\rho \epsilon_0} \frac{\delta^{-L}-1}{\delta^{-1} - 1}
+ \frac{L C_5 \Delta_0(z_0)}{\rho \epsilon_0} \frac{(r^2/\delta)^L - 1}{r^2/\delta - 1} + L
\end{align*}

Denote $G(k)$ the value of the smoothed gap we have after a total of $k$ inner iterations.
We know that there exists a constant $c$ such that $G(c (r^2/\delta)^L) \leq \delta^L$.
If $r = \delta^{-\alpha/2}$, we get  $G(c (\delta^{-\alpha-1})^L) \leq \delta^L$.

Hence, to have $\delta^L \leq \epsilon$, we need $L = \frac{\ln(\epsilon)}{\ln(\delta)}$ outer iterations
and thus a total number of inner iterations at most equal to $c (\delta^L)^{-\alpha-1} = c \epsilon^{-\alpha-1}$.
\end{proof}

\section{Numerical evaluation}

In order to check whether our method has a practical advantage or not, we performed some numerical comparisons.
We performed experiments on a computer with 8 11th Gen Intel(R) Core(TM) i7-1165G7 @ 2.80GHz CPU processors and 16GB RAM. The source code can be found on \url{https://perso.telecom-paristech.fr/ofercoq/Software.html}. It uses the generic primal-dual coordinate descent solver developed in~\cite{fercoq2021generic}.

\subsection{Regularized least squares}

\begin{figure}[htbp]
	\centering
	
	\includegraphics[width=0.7\linewidth, trim=0 2ex 0 8ex, clip]{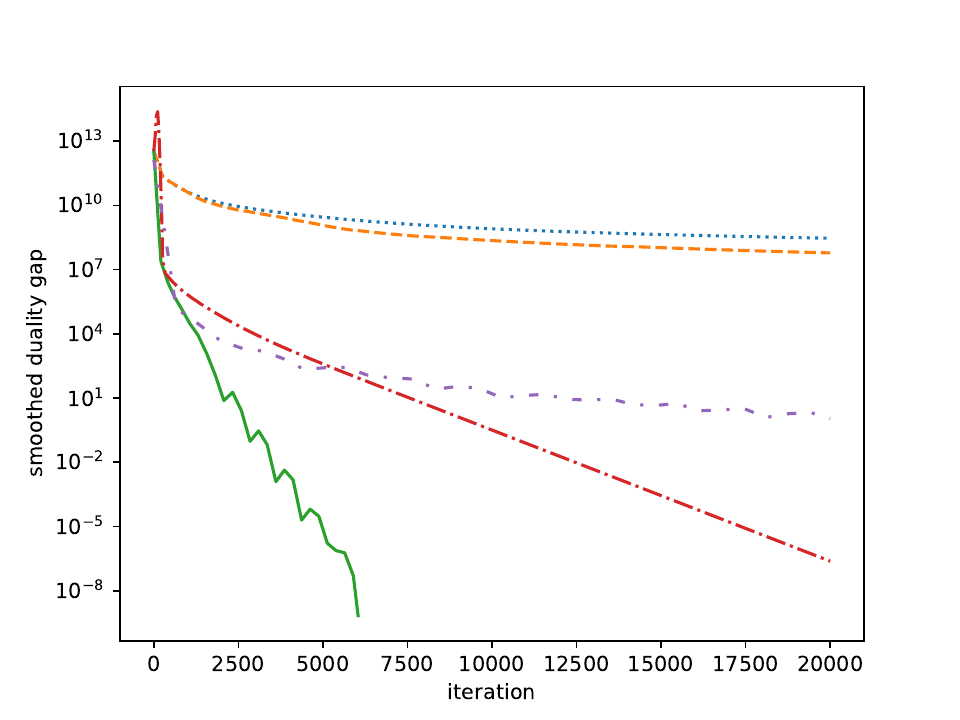}
	\caption{Comparison of various adaptive step sizes strategies for PDHG. Quadratic problem: $\lambda = 10^{-3}$, $A$ and $b$ given in the \texttt{a1a} dataset \cite{chang2011libsvm} for $\min_x \frac 12 \|Ax-b\|^2_2 + \frac{\lambda}{2n}\|x\|^2_2$. We initialize the step size with a factor 1000 compared to the optimal step sizes given in Section~\ref{sec:linconv}. Dotted blue line: constant step sizes. Dashed orange line: Alg. \ref{alg:adap_fercoq} based on rate estimation using residual norm and with $\alpha_0 = 0$. Dash-dotted red line: Goldstein et al.'s adaptive step sizes (Alg. \ref{alg:goldstein}). Solid green line: Alg. \ref{alg:adap_fercoq} combining Goldstein et al's and our step size adaptation. Loosely dash-dotted purple line: Restarted FISTA with optimal restart period.}
	\label{fig:quadratic}
\end{figure}

On Figure~\ref{fig:quadratic}, we look at the behavior of Algorithm~\ref{alg:adap_fercoq} on a quadratic problem: $\ell_2$ regularized least squares. Note that Algorithm~\ref{alg:adap_fercoq} has been designed for quadratic problems. We can see that the conclusions of Figure~\ref{fig:toy_problem} drawn on a toy problem still hold: Goldstein et al.'s step size adaptation rule leads to a fair convergence rate but it can still be significantly improved. However, trying to monitor the convergence rate is too slow when step sizes were initially badly set.

\subsection{Linear program}

On Figure~\ref{fig:sparse_svm}, we consider a linear program. As expected, the adaptive step size rules have no influence on the asymptotic rate of convergence. However, they can reduce or increase the length of the active set discovery phase.
We note that in this case, averaging and restarting PDHG (with legend RAPDHG) greatly improves the speed of convergence \cite{fercoq2022quadratic,applegate2023faster}.

\begin{figure}[htbp]
	\centering
	
	\includegraphics[width=0.7\linewidth, trim=0 2ex 0 8ex, clip]{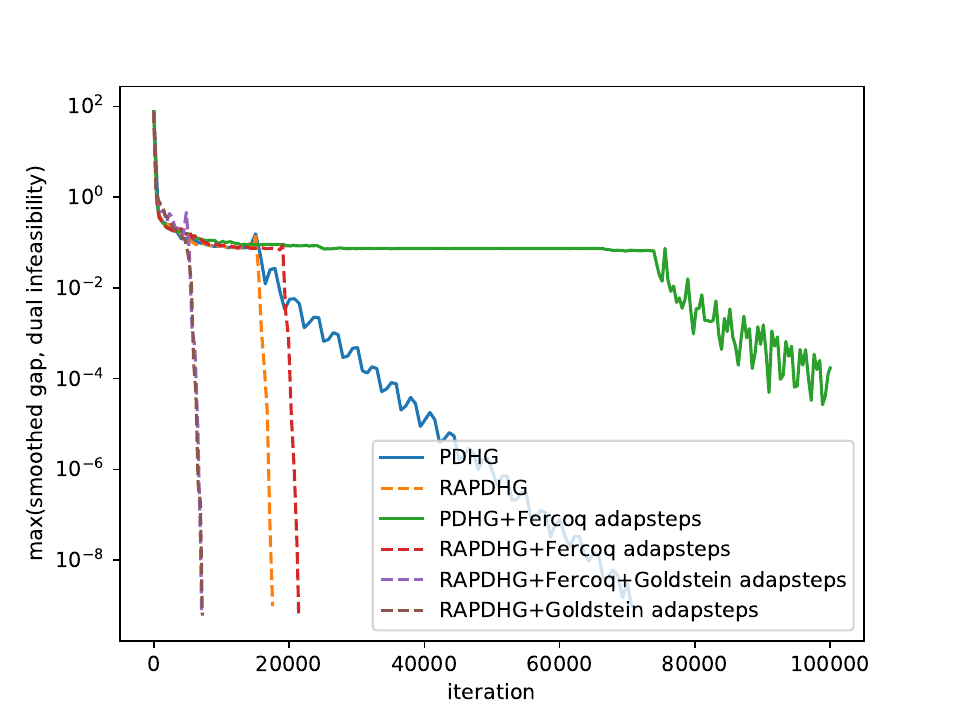}
	\caption{Comparison of various adaptive step sizes strategies for PDHG. Sparse SVM problem (linear program): $A$ and $b$ given in the \texttt{a1a} dataset for $\min_w \sum_i \max(0, 1 - y_i(a_i w)) + ||w||_1$. Averaging and restarting PDHG, with legend RAPDHG and dashed lines, improves significantly the rate \cite{fercoq2022quadratic}. However, changing the step sizes of PDHG has no influence on the rate. Trying to monitor the rate can even be detrimental to the transient active set discovery phase (green and red curves). On this problem, it looks like Goldstein et al.'s adaptive step sizes reduce the length of the transient phase. Moreover, combining both step size strategies as in Alg.~\ref{alg:adap_fercoq} gives two identical convergence profiles.}
	\label{fig:sparse_svm}
\end{figure}

\subsection{TV-L1}

Next, we turn to larger problems and we use PURE-CD (Algorithm~\ref{alg:pure-cd}) to solve them. We first consider a TV-L1 problem which can be written as 
\begin{align*}
\min_x \lambda \|x - I\|_1 + \|Dx\|_{2,1}
\end{align*}
where $I$ is an image and $D$ is the 2D discrete gradient, $\|z\|_{2,1} = \sum_{p \in P} \sqrt{z_{p,1}^2 + z_{p,2}^2}$ and $\lambda = 1.9$. We used the cameraman image (256 x 256 pixels) for the experiment. On Figure~\ref{fig:tvl1_purecd}, we can see that both adaptive step size techniques significantly outperform the default constant step size (\eqref{eq:tau_from_s}\eqref{eq:sigma_from_s} with $s=1$).
\begin{figure}
	\centering
\includegraphics[width=0.7\linewidth]{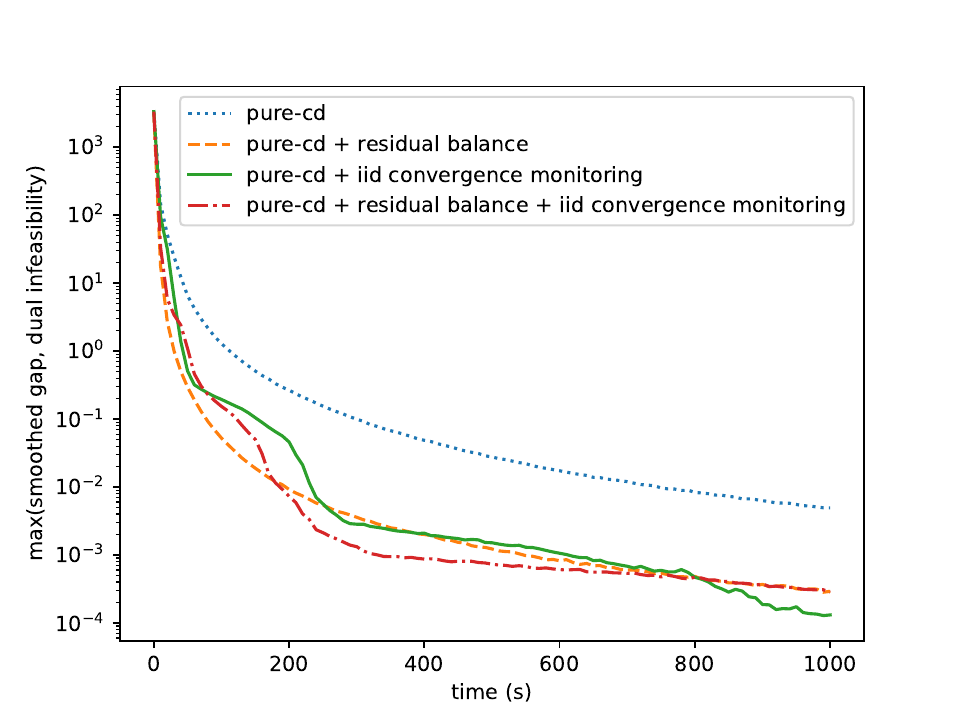}
\caption{Smoothed duality gap as a function of time for the resolution of the TV-L1 problem. We are comparing constant step-size PURE-CD with versions where we adapt step sizes based on residual balance or convergence monitoring (i.i.d. model). On this problem, both adaptive step techniques behave similarly, and they significantly outperform the default constant step size.}
\label{fig:tvl1_purecd}
\end{figure}

\subsection{TV-L2}

\begin{figure}
	\centering
\includegraphics[width=0.7\linewidth]{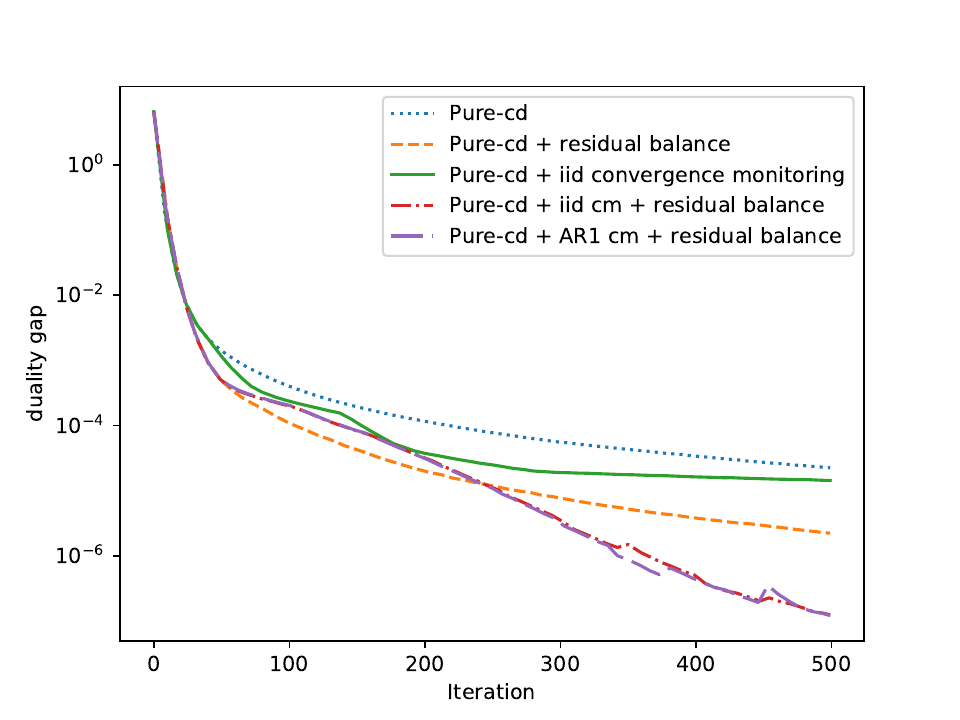}
\caption{Duality gap as a function of iteration for the TV-L2 problem. We compare default step sizes (\eqref{eq:tau_from_s}\eqref{eq:sigma_from_s} with $s=1$), residual balance (Prop. \ref{prop:estimate_subgradient_purecd}), convergence monitoring (Alg. \ref{alg:cm_pure_cd}) and the combination of both of them.}
\label{fig:tvl2_purecd}
\end{figure}

%\begin{figure}
%\centering
%
%\includegraphics[width=0.7\linewidth, trim=0 2ex 0 8ex, clip]{compareMethodsTVdenoising}
%\caption{TV-denoising problem $\min_x \frac 12 \|x - I\|_2^2 + \lambda \|Dx\|_{2,1}$ where $I$ is a 100$\times$100 pixels image of a feather and $\lambda = 0.01$. }
%\end{figure}

The second large scale experiment is a TV-L2 denoising problem. We consider the image $I$ of a hen (574 x 650 pixels) and we solve the optimization problem
\begin{align*}
\min_x \lambda \|x - I\|_2^2 + \|Dx\|_{2,1}
\end{align*}

We can see on Figure \ref{fig:tvl2_purecd} a behavior similar to the previous experiment. Residual balance does help in choosing good step sizes but they are not optimal. Convergence monitoring may be less reactive but leads to better step sizes. 
In particular, initializing convergence monitoring after some iterations where we try to balance the residuals gives a quite nice solution. We note that the stochastic model for the rate seems to have a minor influence on the algorithm.

\subsection{Square-root Lasso problem}

Finally, we consider a square-root lasso problem
\begin{align*}
\min_{x \in \mathbb R^d} \|Ax - b\|_2 + \lambda \|x\|_1
\end{align*}
We chose $A$ and $b$ from the Leukemia dataset \cite{chang2011libsvm} and $\lambda = \frac{\|A^\top b\|_\infty}{1.1 \|b\|_2}$. 
We first can see on Figure~\ref{fig:square-root-lasso} that using a randomized algorithm indeed gives a much faster algorithm that deterministic PDHG, even if Alg.~\ref{alg:adap_fercoq} does help. For this problem, the default step size of PURE-CD is already quite good. Convergence monitoring only validates this step size and does not change it. However, our version of residual balance is not well set for this problem: it chooses a very step size that leads to a very slow algorithm. Fortunately, our convergence monitoring method can detect that the initial step size was better and recover the fast behavior.

\begin{figure}
	\centering
	\includegraphics[width=0.7\linewidth]{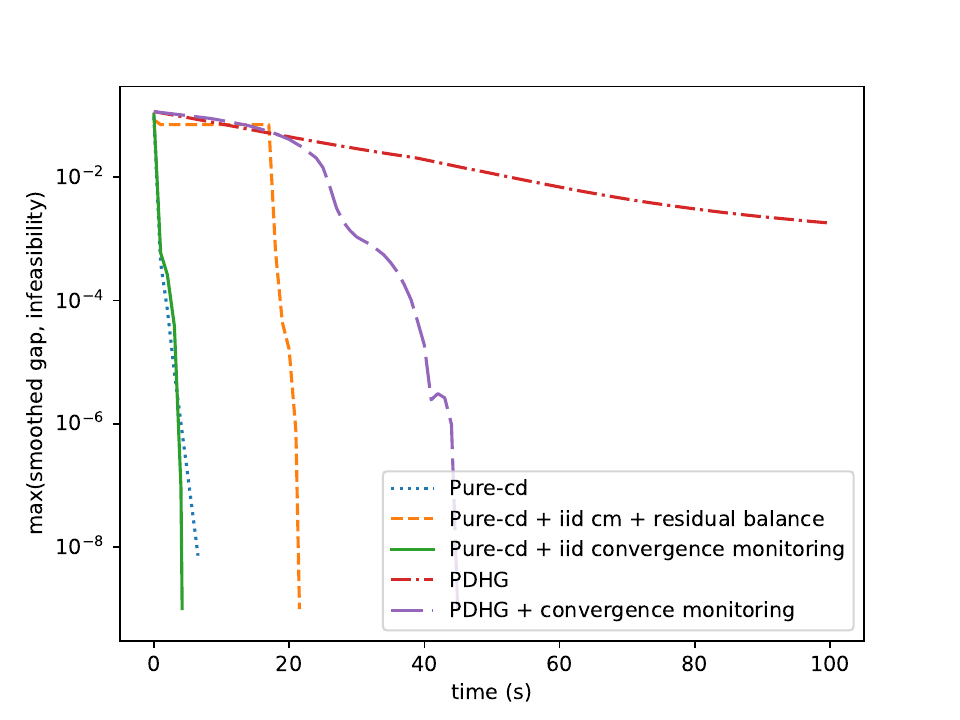}
	\caption{Comparison of PURE-CD and PDHG together with their adaptive step size versions on a square-root lasso problem.}
	\label{fig:square-root-lasso}
\end{figure}

\section{Conclusion}

In this paper, we studied adaptive step sizes for PDHG beyond residual balance. 
Our proposal is to monitor convergence estimates in order to compare tentative step sizes and select the better one. The challenge lies in designing fast and accurate convergence rate estimates. We showed on numerical experiments that using an appropriate norm and detecting periodic features in the instantaneous rate estimate, we obtain an algorithm which is competitive against the state of the art. 

We also developed convergence monitoring based adaptive step sizes for a randomized version of the algorithm, called PURE-CD, with promising results. In the randomized case, we proposed two simple models for the instantaneous rate estimate and we delay the decision of changing step sizes to the point where we can see a statistically significant difference in the estimated rates. We could combine it with a residual balance technique and our technique allows us to compare and choose the best step size between the initial one and the one that leads to residual balance.

Finally, we proposed a new convergence guarantee for adaptive step sizes which allows for significant changes even after many iterations and does not limit a priori the amount of change in step size magnitude.

Future works may focus on the following aspects:
\begin{itemize}
	\item Like previous works, we did not prove that our adaptive step size technique indeed leads to a faster algorithm. We only showed that our methods do not prevent convergence.
	\item A study of convergence properties beyond the quadratic case may help monitor convergence before activity identification occurs and thus improve the speed of the algorithm in the initial phases. 
\end{itemize}

\section*{Acknowlegments}

This work was supported by the Agence National de la Recherche grant ANR-20-CE40-0027, Optimal Primal-Dual Algorithms (APDO). 
%I would like also to thank Rustem Islamov for conducting preliminary research related to this paper.

\bibliographystyle{alpha}
\bibliography{literature}

\newcommand{\etalchar}[1]{$^{#1}$}
\begin{thebibliography}{TDAFC20}

\bibitem[ADH{\etalchar{+}}21]{applegate2021practical}
David Applegate, Mateo D{\'\i}az, Oliver Hinder, Haihao Lu, Miles Lubin,
  Brendan O'Donoghue, and Warren Schudy.
\newblock Practical large-scale linear programming using primal-dual hybrid
  gradient.
\newblock {\em Advances in Neural Information Processing Systems},
  34:20243--20257, 2021.

\bibitem[AFC20]{alacaoglu2020random}
Ahmet Alacaoglu, Olivier Fercoq, and Volkan Cevher.
\newblock Random extrapolation for primal-dual coordinate descent.
\newblock In {\em International conference on machine learning}, pages
  191--201. PMLR, 2020.

\bibitem[AHLL23]{applegate2023faster}
David Applegate, Oliver Hinder, Haihao Lu, and Miles Lubin.
\newblock Faster first-order primal-dual methods for linear programming using
  restarts and sharpness.
\newblock {\em Mathematical Programming}, 201(1-2):133--184, 2023.

\bibitem[BC11]{bauschke2011convex}
Heinz~H Bauschke and Patrick~L Combettes.
\newblock {\em Convex analysis and monotone operator theory in Hilbert spaces}.
\newblock Springer, 2011.

\bibitem[CDE{\etalchar{+}}23]{chambolle2023stochastic}
Antonin Chambolle, Claire Delplancke, Matthias~J Ehrhardt, Carola-Bibiane
  Sch{\"o}nlieb, and Junqi Tang.
\newblock Stochastic primal dual hybrid gradient algorithm with adaptive
  step-sizes.
\newblock {\em arXiv preprint arXiv:2301.02511}, 2023.

\bibitem[CERS18]{chambolle2018stochastic}
Antonin Chambolle, Matthias~J Ehrhardt, Peter Richt{\'a}rik, and Carola-Bibiane
  Schonlieb.
\newblock Stochastic primal-dual hybrid gradient algorithm with arbitrary
  sampling and imaging applications.
\newblock {\em SIAM Journal on Optimization}, 28(4):2783--2808, 2018.

\bibitem[CKCH23]{condat2023proximal}
Laurent Condat, Daichi Kitahara, Andr{\'e}s Contreras, and Akira Hirabayashi.
\newblock Proximal splitting algorithms for convex optimization: A tour of
  recent advances, with new twists.
\newblock {\em SIAM Review}, 65(2):375--435, 2023.

\bibitem[CL11]{chang2011libsvm}
Chih-Chung Chang and Chih-Jen Lin.
\newblock Libsvm: a library for support vector machines.
\newblock {\em ACM transactions on intelligent systems and technology (TIST)},
  2(3):1--27, 2011.

\bibitem[Con13]{condat2013primal}
Laurent Condat.
\newblock A primal--dual splitting method for convex optimization involving
  lipschitzian, proximable and linear composite terms.
\newblock {\em Journal of optimization theory and applications},
  158(2):460--479, 2013.

\bibitem[CP11]{chambolle2011first}
Antonin Chambolle and Thomas Pock.
\newblock A first-order primal-dual algorithm for convex problems with
  applications to imaging.
\newblock {\em Journal of mathematical imaging and vision}, 40:120--145, 2011.

\bibitem[DY16]{davis2016convergence}
Damek Davis and Wotao Yin.
\newblock Convergence rate analysis of several splitting schemes.
\newblock {\em Splitting methods in communication, imaging, science, and
  engineering}, pages 115--163, 2016.

\bibitem[FB19]{fercoq2019coordinate}
Olivier Fercoq and Pascal Bianchi.
\newblock A coordinate-descent primal-dual algorithm with large step size and
  possibly nonseparable functions.
\newblock {\em SIAM Journal on Optimization}, 29(1):100--134, 2019.

\bibitem[Fer21]{fercoq2021generic}
Olivier Fercoq.
\newblock A generic coordinate descent solver for non-smooth convex
  optimisation.
\newblock {\em Optimization Methods and Software}, 36(6):1202--1222, 2021.

\bibitem[Fer22]{fercoq2022quadratic}
Olivier Fercoq.
\newblock Quadratic error bound of the smoothed gap and the restarted averaged
  primal-dual hybrid gradient.
\newblock {\em arXiv preprint arXiv:2206.03041}, 2022.

\bibitem[GEB13]{goldstein2013adaptive}
Tom Goldstein, Ernie Esser, and Richard Baraniuk.
\newblock Adaptive primal dual optimization for image processing and learning.
\newblock In {\em Proceedings of the 6th NIPS Workshop on Optimization for
  Machine Learning}, 2013.

\bibitem[GLY15]{goldstein2015adaptive}
Tom Goldstein, Min Li, and Xiaoming Yuan.
\newblock Adaptive primal-dual splitting methods for statistical learning and
  image processing.
\newblock {\em Advances in neural information processing systems}, 28, 2015.

\bibitem[Hag21]{hager2021applied}
William~W Hager.
\newblock {\em Applied numerical linear algebra}.
\newblock SIAM, 2021.

\bibitem[IF22]{islamov2021pdhg}
Rustem Islamov and Olivier Fercoq.
\newblock Efficiency of primal-dual hybrid gradient method.
\newblock Master's thesis, 2022.

\bibitem[LFP18]{liang2018local}
Jingwei Liang, Jalal Fadili, and Gabriel Peyr{\'e}.
\newblock Local linear convergence analysis of primal--dual splitting methods.
\newblock {\em Optimization}, 67(6):821--853, 2018.

\bibitem[LP18]{latafat2018primal}
Puya Latafat and Panagiotis Patrinos.
\newblock Primal-dual proximal algorithms for structured convex optimization: A
  unifying framework.
\newblock {\em Large-Scale and Distributed Optimization}, pages 97--120, 2018.

\bibitem[LY22]{lu2022infimal}
Haihao Lu and Jinwen Yang.
\newblock On the infimal sub-differential size of primal-dual hybrid gradient
  method.
\newblock {\em arXiv preprint arXiv:2206.12061}, 2022.

\bibitem[Nes03]{nesterov2003introductory}
Yurii Nesterov.
\newblock {\em Introductory lectures on convex optimization: A basic course},
  volume~87.
\newblock Springer Science \& Business Media, 2003.

\bibitem[TDAFC20]{tran2020adaptive}
Quoc Tran-Dinh, Ahmet Alacaoglu, Olivier Fercoq, and Volkan Cevher.
\newblock An adaptive primal-dual framework for nonsmooth convex minimization.
\newblock {\em Mathematical Programming Computation}, 12(3):451--491, 2020.

\bibitem[TDFC18]{tran2018smooth}
Quoc Tran-Dinh, Olivier Fercoq, and Volkan Cevher.
\newblock A smooth primal-dual optimization framework for nonsmooth composite
  convex minimization.
\newblock {\em SIAM Journal on Optimization}, 28(1):96--134, 2018.

\bibitem[TT83]{tiao1983consistency}
George~C Tiao and Ruey~S Tsay.
\newblock Consistency properties of least squares estimates of autoregressive
  parameters in arma models.
\newblock {\em The Annals of Statistics}, pages 856--871, 1983.

\bibitem[V{\~u}13]{vu2013splitting}
Bang~C{\^o}ng V{\~u}.
\newblock A splitting algorithm for dual monotone inclusions involving
  cocoercive operators.
\newblock {\em Advances in Computational Mathematics}, 38:667--681, 2013.

\bibitem[WF24]{walwil:hal-04501394}
Iyad Walwil and Olivier Fercoq.
\newblock {The Smoothed Duality Gap as a Stopping Criterion}.
\newblock working paper or preprint, March 2024.

\end{thebibliography}
	
\end{document}